%% file: bary.tex
\documentclass[11pt,letterpaper]{article}
\input{bary_preamble}

\usepackage[margin=1in]{geometry}

\makeatletter
\def\blfootnote{\gdef\@thefnmark{}\@footnotetext}
\makeatother

\begin{document}
	\title{Wasserstein barycenters can be computed in polynomial time in fixed dimension}
	\author{Jason M. Altschuler \and Enric Boix-Adser\`a}
	\date{\today}
	\maketitle
	
\blfootnote{The authors are with the Laboratory for Information and Decision Systems (LIDS), Massachusetts Institute of Technology, Cambridge MA 02139. Work partially supported by NSF Graduate Research Fellowship 1122374, a TwoSigma PhD Fellowship, and a Siebel PhD Fellowship.}

\begin{abstract}
		Computing Wasserstein barycenters is a fundamental geometric problem with widespread applications in machine learning, statistics, and computer graphics. However, it is unknown whether Wasserstein barycenters can be computed in polynomial time, either exactly or to high precision (i.e., with $\polylog(1/\eps)$ runtime dependence). This paper answers these questions in the affirmative for any fixed dimension. Our approach is to solve an exponential-size linear programming formulation by efficiently implementing the corresponding separation oracle using techniques from computational geometry. 
\end{abstract}

	\section{Introduction}\label{sec:intro}

Given discrete probability distributions $\mu_1, \dots, \mu_k$ supported on $\R^d$ and a vector $\lambda \in \R^k$ of non-negative weights summing to $1$, the corresponding \emph{Wasserstein barycenters} are the probability distributions $\nu$ minimizing
\begin{align}
	\argmin_{\nu} \sum_{i=1}^k \lambda_i \cW(\mu_i,\nu), \label{eq:intro:bary}
\end{align}
where above $\cW(\cdot,\cdot)$ denotes the squared $2$-Wasserstein distance~\citep{AguCar11}. Wasserstein barycenters provide a natural extension of the notion of averaging points to the notion of averaging point clouds. Importantly, they naturally inherit the ability of optimal transportation to capture geometric properties of the data. 
\par This desirable property has led to the widespread use of Wasserstein barycenters in many applications. Applications in statistics and machine learning include for instance the $n$-coupling problem~\citep{ruschendorf2002n}, constrained clustering~\citep{CutDou14,ho2017multilevel}, fusing measurements from partial sensors~\citep{elvander2020multi}, and fusing measurements for scalable Bayesian learning~\citep{SriLiDun18}. Applications in image processing and computer graphics include for instance texture mixing~\citep{RabPeyDel11} and shape interpolation~\citep{solomon2015convolutional}. For further applications, see the surveys~\citep{PeyCut17,panaretos2019statistical}.

\paragraph*{Open problem: computing barycenters in polynomial time.} Despite considerable algorithmic work, it is an open problem (e.g.,~\citep{Bor17}) whether Wasserstein barycenters between discrete distributions can be exactly computed in polynomial time in the input size. A highly related open problem is whether Wasserstein barycenters can be computed to high accuracy, i.e., whether an $\eps$-additively approximate solution for~\eqref{eq:intro:bary} can be computed in time that is polynomial in the input size and $\log(1/\eps)$. This paper answers these questions in the affirmative for any fixed dimension $d$.

\par Previous methods require time that depends polynomially on $1/\eps$ in order to compute $\eps$-approximate barycenters. This means that in practice they can only solve to a few digits of precision. See the prior work section for details. In many applications, including nearly all of those mentioned above, Wasserstein barycenters are used as a subroutine in a larger pipeline to solve downstream data science tasks. Thus, high-precision algorithms are important for downstream performance and to avoid error propagation, especially in applications which require multiple barycenter computations.

\paragraph{Key obstacle.} The well-documented key obstacle is that \emph{a priori}, there are $n^k$ candidate atoms for the barycenter's support, where $n$ is (an upper bound on) the number of atoms in each $\mu_i$.
While there always exists a barycenter with $\poly(n,k)$ atoms, finding which atoms those are requires pruning the $n^k$ exponentially many candidates.

\subsection{Prior work}\label{sec:prev}

The literature on computing Wasserstein barycenters is extensive and rapidly growing.\footnote{We mention in passing that an orthogonal line of work aims to compute barycenters of \emph{continuous} distributions (typically restricted to Gaussian distributions so that both $\mu_i$ and $\nu$ have compact representations for computational purposes). See, e.g.,~\citep{chewi2020gradient,AlvEtAl16}.} Existing algorithms can be partitioned into two categories, depending on how they handle optimizing over the $n^k$ exponentially many candidate atoms for the barycenter support.

\paragraph*{Fixed-support.} Most existing algorithms work around this exponential complexity by making a ``fixed-support approximation'': they assume that the barycenter is supported on a small guessed set of points, and then only optimize over the corresponding masses. This reduces the barycenter problem to a polynomial-size LP, which can then be solved efficiently using out-of-the-box LP solvers, or alternatively using specialized methods such as entropic regularization; see, e.g.,~\citep{CutDou14,BenCarCut15,solomon2015convolutional,CarObeOud15,staib2017parallel,KroDviDvuetal19,janati2020debiased,lin2020fixed} among many others. However, the key issue with fixed-support algorithms is that guessing a reasonable support set for the barycenter requires $\eps$-covering the space. Specifically, these algorithms require fixing the support to roughly $(R/\eps)^d$ points in order to get an $\eps$-additive approximation to the barycenter\footnote{An alternative is to restrict to the union of the supports of $\mu_1,\dots, \mu_k$, which has only $nk$ points. However, this cannot get arbitrarily close approximations~\citep{Bor17}.
},
where $R$ is a bound on the squared diameter of the supports of the input distributions. This results in $\poly(n,k,R/\eps)$ final runtimes in constant dimension $d$.

\par In contrast, our proposed algorithm has $\poly(n,k,\log (R/\eps))$ runtime which, critically, has polylogarithmic dependence on $R/\eps$.\footnote{This means that our algorithm solves the barycenter problem exactly in polynomial time, whereas previous algorithms require \emph{pseudo-polynomial} time. This is because solving the barycenter problem exactly  requires $\eps$ to be exponentially small in the bit-complexity of the input.} In practice, this means that our algorithm can often solve up to machine precision, whereas fixed-support algorithms can only solve up to a few digits of precision---see \S\ref{sec:experiments} for experiments.

\paragraph*{Free-support.} 
Achieving our $\poly(n,k,\log (R/\eps))$ runtime precludes making a fixed-support approximation. Such algorithms are called ``free-support algorithms''. The key obstacle is that, as mentioned before, this requires optimizing over the support set of the barycenters, which \emph{a priori} can only be restricted to $n^k$ candidate points. All previous free-support algorithms either run in exponential time, or are heuristics without provable guarantees; see, e.g.,~\citep{CutDou14,luise2019sinkhorn}.
\par In contrast, we show how to optimize over these $n^k$ candidate points in polynomial time by exploiting the geometric structure of their configuration.

\subsection{Contribution} 
We give the first algorithm that, in any fixed dimension $d$, solves the Wasserstein barycenter problem exactly or to high precision in polynomial time. For simplicity of notation, throughout $d$ is a constant; the running time for fixed $d$ is $(nk)^{d}$ times a polynomial in the input size.

\begin{theorem}[Computing high-precision barycenters]\label{thm:mainbary}
	There is an algorithm that, given $k$ distributions each supported on $n$ atoms in the ball of squared radius $R$ in $\R^d$, a weight vector $\lambda$, and an accuracy $\eps> 0$, computes an $\eps$-additively approximate Wasserstein barycenter in $\poly(n,k,\log(R/\eps))$ time. Moreover, this barycenter has support size at most $nk-k+1$.
\end{theorem}

\begin{theorem}[Computing exact barycenters]
	If the weight vector and distributions are represented with $\log U$ bits of precision, then an exact barycenter can be found in 
	$\poly(n,k,\log U)$ time. Moreover, this barycenter has support size at most $nk-k+1$.
	\label{thm:exactbary}
\end{theorem}

Our algorithm is described in \S\ref{sec:alg}.
Briefly, the high-level idea is as follows. Our starting point is a well-known LP reformulation of the Wasserstein barycenter problem as a Multimarginal Optimal Transport (MOT) problem, recalled in the preliminaries section. This is an exponential-size LP with $n^k$ variables, one for each candidate atom. Nevertheless, we show that a sparse solution can be computed in polynomial time. Our approach consists of two steps. First, by leveraging tools from combinatorial optimization and exploiting the special structure of the MOT LP, we show that this problem can be solved efficiently if one can efficiently implement the separation oracle for the LP dual of MOT. Second, by leveraging tools from computational geometry such as power diagrams and the complexity of hyperplane arrangements, we show how to efficiently implement this separation oracle.

In addition to its polynomial runtime, our algorithm has two additional properties that may be useful in downstream applications. First, the outputted barycenter $\nu$ has small support of $O(nk)$ size, which is much smaller than the \emph{a priori} $n^k$ bound on the support size. In particular, the support size of $\nu$ is at most the maximal sparsity of any vertex of the transportation polytope between $\mu_1, \dots, \mu_k$---which is at most $nk-k+1$.
Note that Theorem~\ref{thm:exactbary} is not at odds with the NP-hardness of finding the sparsest barycenter~\citep{BorPat19}: indeed, our algorithm outputs a solution that albeit sparse is not necessarily the sparsest.
Second, as a by-product, our algorithm also produces sparse solutions to the optimal transport problems $\cW(\mu_i,\nu)$ that are non-mass-splitting maps from $\nu$ to $\mu_i$. Among other benefits, this enables easy visualization and interpretability of the results---in comparison to entropic-regularization based approaches which produce ``blurry'' dense maps.

Although the focus of this work is theoretical, we also provide preliminary numerical experiments in \S\ref{sec:experiments} demonstrating that a slight variant of our algorithm can provide high-precision solutions at previously intractable problem sizes.

Finally, in \S\ref{sec:disc}, we briefly mention that the techniques we develop in this paper extend to solving several related problems. In particular, this gives the first polynomial-time algorithms for computing geometric medians with respect to the $1$-Wasserstein distance (a.k.a. Earth Mover's distance) over any of the popular ground metrics $\ell_1$, $\ell_2$, or $\ell_{\infty}$.

\section{Preliminaries}\label{sec:prelim}

This section is organized as follows. First, in \S\ref{ssec:prelim:not}, we establish our notation, which is mostly standard. Then in \S\ref{ssec:prelim:mot}, \S\ref{ssec:prelim:ellipsoid}, \S\ref{ssec:prelim:cg}, respectively, we recall relevant background from the machine learning, combinatorial optimization, and computational geometry literatures---namely, background about LP formulations of Wasserstein barycenters, algorithms for solving exponential-size LP, and algorithms for manipulating power diagrams.

\subsection{Notation}\label{ssec:prelim:not}

The set $\{1, \dots, n\}$ is denoted by $[n]$. The $k$-fold tensor product space $\Rn \otimes \dots \otimes \Rn$ is denoted by $\Rntk$, and similarly for $\Rpntk$. For shorthand, we often denote a tuple $(j_1, \dots, j_k) \in [n]^k$ by $\jvec$. The $i$-th marginal, $i \in [k]$, of a tensor $P \in (\RR^n)^{\otimes k}$ is denoted by the vector $m_i(P) \in \RR^n$, and has entries 
$[m_i(P)]_{\ell} := \sum_{\jvec \in [n]^k : j_i = \ell} P_{\jvec}$.
The transportation polytope between $\mu_1,\ldots,\mu_k$ is the set of joint distributions with one-dimensional marginal distributions $\mu_1,\ldots,\mu_k$, and is identified with the set $\Coup := \{P \in \Rpntk \,:\, m_i(P) = \mu_i, \, \forall i \in [k]\}$, where we abuse notation slightly by identifying $\mu_i$ with its vector of probabilities (in any order). The closure of a set $E \subset \RR^d$ (with respect to the standard topology) is denoted by $\overline{E}$. Throughout, we assume without loss of generality that each $\lambda_i$ is strictly positive, since otherwise $\mu_i$ does not affect the barycenter (see equation~\eqref{eq:intro:bary}).

\subsection{LP formulation of Wasserstein barycenters}\label{ssec:prelim:mot}

Our starting point is the known fact (see, e.g.,~\citep{AguCar11,BenCarCut15,anderes2016discrete}) that a barycenter $\nu$ can be found by solving the Multimarginal Optimal Transport problem
\begin{align}\min_{P \in \Coup} \langle P, C \rangle,
	\tag{MOT}
	\label{MOT-P}
\end{align}
for the cost tensor $C \in \Rntk$ with entries \begin{align}C_{\jvec} &= \min_{y \in \RR^d} \sum_{i=1}^k \lambda_i \|x_{i,j_i} - y\|^2,
	\label{C-WASS}
\end{align}
or equivalently, $C_{\jvec} = \sum_{i=1}^k \lambda_i \|x_{i,j_i} - \sum_{\ell=1}^k \lambda_{\ell} x_{\ell, j_{\ell}}\|^2$ by optimality of $y = \sum_{\ell=1}^k \lambda_{\ell} x_{\ell,j_{\ell}}$. Specifically, the reduction from the Wasserstein barycenter problem to the LP~\eqref{MOT-P} is as follows.

\begin{lemma}\label{lem:bary-to-mot}
	If $P \in \Coup$ is an optimal solution to~\eqref{MOT-P}, then the pushforward of $P$ under the map $(X_1, \dots, X_k) \mapsto \sum_{i=1}^k \lambda_i X_i$ is an optimal barycenter $\nu$. Furthermore, the support size of $\nu$ is at most the support size of $P$, and also the coupling $(\sum_{i=1}^k \lambda_i X_i,X_j)$ is a non-mass-splitting map that solves the Optimal Transport problem from $\nu$ to $\mu_j$.
\end{lemma}

Notice that applying this pushforward map $(X_1, \dots, X_k) \mapsto \sum_{i=1}^k \lambda_i X_i$ in order to compute $\nu$ from $P$ requires only $O(skd)$ arithmetic operations, where $s$ denotes the support size of $P$. In particular, this takes polynomial time if $s$ is of polynomial size. Therefore it suffices to compute a sparse solution $P$ of the LP~\eqref{MOT-P}.  

\par Note that the solution $P$ is guaranteed to be sparse if it is a vertex solution. Indeed, since~\eqref{MOT-P} is a standard-form LP whose constraints have rank at most $nk-k+1$, each vertex solution has at most $nk-k+1$ non-zero entries. 

\begin{lemma}\label{lem:vertex-sparse}
	If $P$ is a vertex of the transportation polytope $\Coup$, then $P$ has at most $nk-k+1$ non-zero entries. 
\end{lemma}

\par An obvious obstacle for computing any solution---let alone a sparse solution---of the LP formulation~\eqref{MOT-P} is that it has $n^k$ exponentially many variables. An LP that will be useful to us in the sequel is its dual
\begin{align} \max_{p_1,\ldots,p_k \in \Rn} \sum_{i=1}^k \langle p_i, \mu_i \rangle 
	\quad \text{subject to} \quad 
	C_{\jvec} - \sum_{i=1}^k [p_i]_{j_i} \geq 0, \;\; \forall \jvec \in [n]^k.
	\tag{MOT-D} \label{MOT-D} 
\end{align}
An attractive property of~\eqref{MOT-D} is that it has only $nk$ polynomially many variables. However, of course this alone is not enough to make solving~\eqref{MOT-D} tractable because it has $n^k$ exponentially many constraints. That is, dualizing has transferred the exponential complexity in the number of variables in~\eqref{MOT-P} to the number of constraints in~\eqref{MOT-D}.

\subsection{Algorithms for exponential-size LP}\label{ssec:prelim:ellipsoid}

It is a classical fact (see, e.g.,~\citep{khachiyan1980polynomial,BerTsi97,GLSbook}) that regardless of the number of constraints, an LP with polynomially many variables can be solved in polynomial time so long as there is a polynomial-time implementation of the separation oracle for its feasibility set. Here we recall the technical details of this fact. 

\par First, we recall the definition of a separation oracle. This definition is simply an algorithmic reformulation of the Separating Hyperplane Theorem, which states that for any convex set $\cK \subseteq \R^N$ and any point $p \in \R^N$, exactly one of two alternatives must hold: either $p \in \cK$, or there exists a hyperplane that separates $p$ from $\cK$ (i.e., there exists a vector $h \in \R^N$ and a scalar $g \in \R$ such that $\langle h,p\rangle \geq g$ and $\langle h,x\rangle < g$ for all $x \in \cK$).

\begin{defin}
	A \emph{separation oracle} for a convex set $\cK$ is an algorithm that given a point $p$, either outputs that $p \in \cK$ or outputs a hyperplane that separates $p$ from $\cK$.
\end{defin}

\par Given a polynomial-time implementation of a separation oracle for a polytope, the Ellipsoid algorithm can solve an LP over that polytope in polynomial time. This result can be found in~\citep{GroLovSch81}.

\begin{theorem}\label{thm:generalellipguarantee} Let $\log U$ be an upper bound on the number of bits needed to represent any entry in $A \in \RR^{M \times N}$, $b \in \RR^M$, or $c \in \RR^N$. Then the Ellipsoid algorithm finds a vertex solution to $\argmin \{c^T x : x \in \cP \}$ in $\poly(N,\log U)$ time and $\poly(N,\log U)$ calls to a separation oracle for the polytope $\cP = \{x \in \R^{N} : Ax \leq b\}$.
\end{theorem}

\subsection{Computational geometry algorithms}\label{ssec:prelim:cg}

A key ingredient in our barycenter algorithm is power diagrams. Here we introduce these objects and some basic facts about their complexity. Although $d$ is a fixed constant in our final results, we state the explicit dependence on the dimension $d$ in these power diagram complexity bounds to highlight how and where our algorithm incurs exponential runtime dependence in $d$.

\begin{defin}The \emph{power diagram} on the spheres $S(z_1,r_1), \ldots, S(z_n,r_n)$ with centers $z_j \in \RR^d$ and radii $r_j \geq 0$ is the cell complex whose cells $E_1,\ldots,E_n$ are given by $$
	E_j = \{y \in \RR^d \, : \, \|z_j - y\|^2 - r_j^2 < \|z_{j'} - y\|^2 - r_{j'}^2, \, \forall j' \neq j\}.
	$$
\end{defin} 

A power diagram ``essentially'' partitions $\R^d$ in the sense that its cells are disjoint and their closures cover $\R^d$. See Figure~\ref{fig:powerdiagexample} for an illustration. Following are two relevant classical facts. The first essentially shows that a power diagram is defined by a small hyperplane arrangement which can moreover be computed efficiently. 

\begin{lemma}\textbf{\emph{(Theorems 1 and 7 of \citep{aurenhammer1987power}, using convex hull algorithm of \citep{chazelle1993optimal})}}
	A power diagram on $n$ spheres in $\RR^d$ has $O(n)$ affine facets of dimension $d-1$. Moreover these facets can be computed in $O((n \log n + n^{\lceil d/2 \rceil}) \cdot \polylog\; U)$ time, where $\log U$ is the number of bits of precision.\label{lem:powdiagramruntime}
\end{lemma}

The second is about hyperplane arrangements. In the sequel this lets us bound the complexity of the ``intersection'' of multiple power diagrams (defined in \S\ref{ssec:alg:step2}).

\begin{lemma}[Theorem 3.3 of \citep{edelsbrunner1986constructing}]
	The cell complex formed by an arrangement of $N$ hyperplanes in $\RR^d$, represented up to $\log U$ bits of precision, can be computed in $ N^d \cdot \polylog(N,U)$ time.  \label{lem:hyperplaneintersectionenumeration}
\end{lemma}

\begin{figure}
	\centering
	\begin{tabular}{cc}
		\ \ \ \ 
		\includegraphics[clip, trim = 3.6cm 1.2cm 3.2cm 1.2cm, scale=0.37  ]{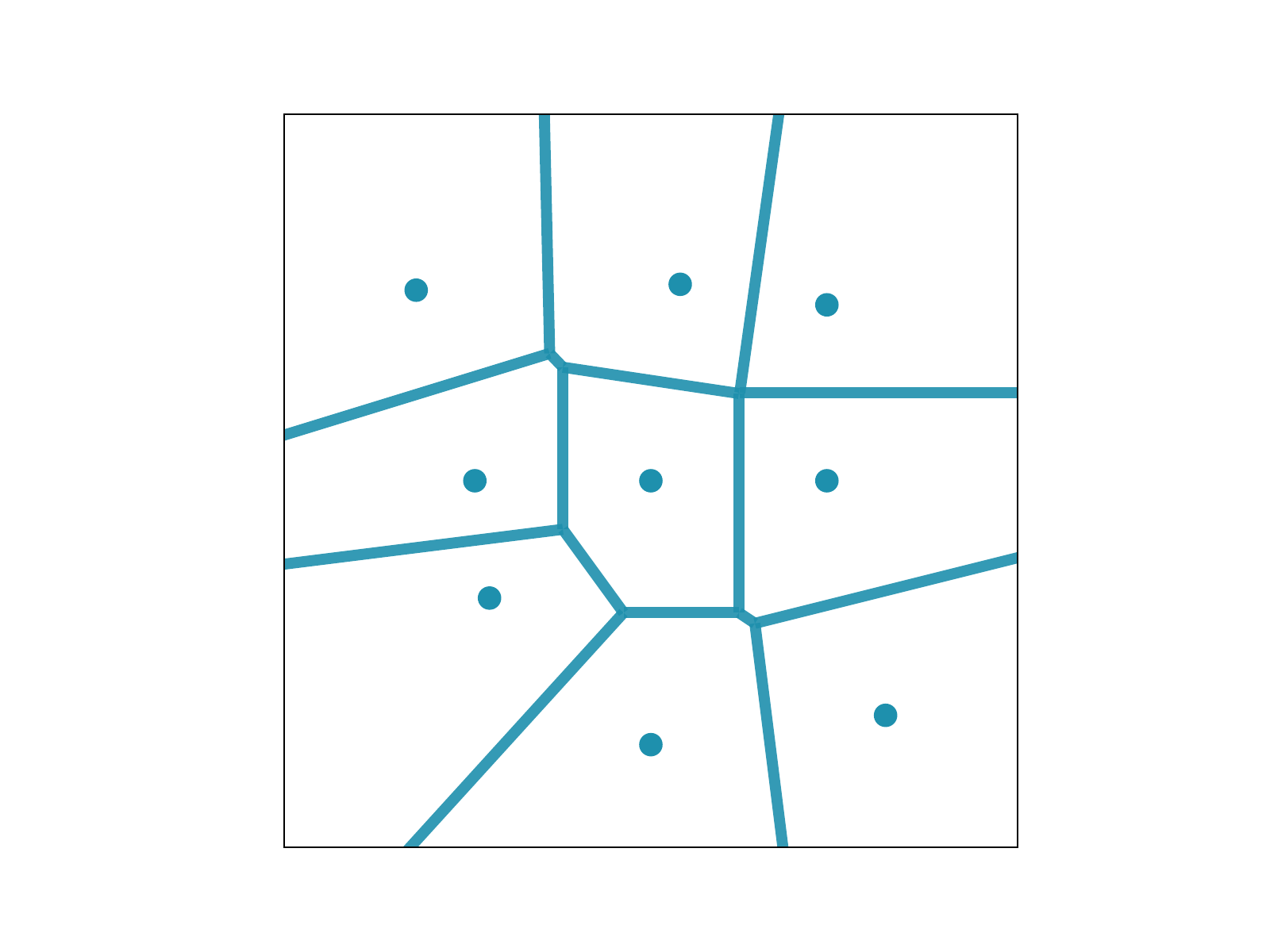}
		\ \ \ \ 
		&
		\ \ \ \ 
		\includegraphics[clip, trim = 3.6cm 1.2cm 3.2cm 1.2cm, scale=0.37]{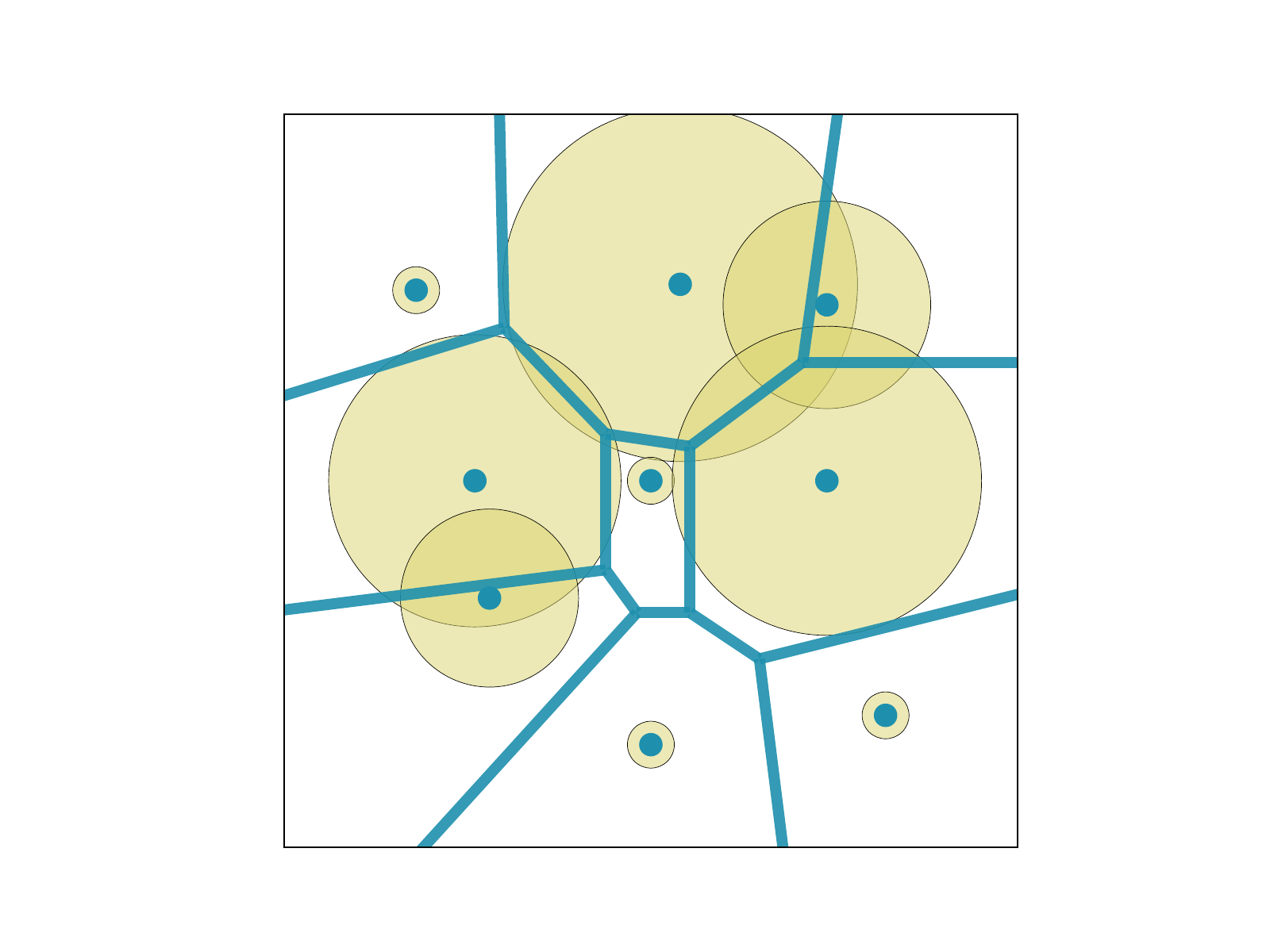}
		\ \ \ \ 
	\end{tabular}
	\caption{Two power diagrams on the same $n = 9$ points with varying weights $w$. Left: all weights are zero (so this is a Voronoi diagram). Right: the weight of a point is indicated by the size of the ball around it. Increasing the weight of a point increases the size of its cell.} \label{fig:powerdiagexample}
\end{figure}

\section{Algorithm}\label{sec:alg}

In this section we describe our algorithm and prove our main results (Theorems~\ref{thm:mainbary} and~\ref{thm:exactbary}). We begin by overviewing the high-level approach. Let us consider the exact solver in Theorem~\ref{thm:exactbary}, as the approximate solver in Theorem~\ref{thm:mainbary} is implemented by exactly solving a rounded problem (see \S\ref{ssec:final} for details).

\par Recall from \S\ref{ssec:prelim:mot} that it suffices to solve the LP formulation~\eqref{MOT-P} of the Wasserstein barycenter problem. However, solving this LP presents a computational obstacle since it has $n^k$ decision variables. Moreover, we desire a sparse solution $P$---rather than a generic solution which has exponentially many non-zero entries---since a polynomial sparsity for $P$ ensures a polynomial support size for the final barycenter $\nu$.

\par The starting point of our approach is recalling the classical fact from \S\ref{ssec:prelim:ellipsoid} that regardless of the number of constraints, an LP with polynomially many variables can be solved in polynomial time so long as the corresponding separation oracle can be implemented in polynomial time. While~\eqref{MOT-P} is not such an LP since it has exponentially many variables, this result applies to its dual~\eqref{MOT-D}. That is, we can efficiently solve~\eqref{MOT-D} so long as we can efficiently implement the corresponding separation oracle.

\par However, two key obstacles remain. First, 
recovering a primal solution is non-trivial in that in general a dual solution does not necessarily ``help'' to find a primal solution~\citep[Exercise 4.17]{BerTsi97}, let alone a sparse primal solution. Second, and most importantly, this approach requires an efficient implementation of the separation oracle for~\eqref{MOT-D}, which does not exist for general Multimarginal Optimal Transport problems (for concrete $\NP$-hard examples see~\citep{AltBoi20hard}). 

\par We solve these issues in two steps:
\begin{enumerate}
	\item \emph{Reduction to separation oracle.} We reduce solving~\eqref{MOT-P} in polynomial time to solving the separation oracle for the dual LP~\eqref{MOT-D} in polynomial time. (Further, we show how to ensure the solution is polynomially sparse.)
	\item \emph{Efficient algorithm for separation oracle.} We use tools from computational geometry to solve the separation oracle for~\eqref{MOT-D} in polynomial time.
\end{enumerate}

At this point, it is worth remarking what special ``structure'' of the exponential-size LP~\eqref{MOT-P} we exploit in order to solve it in polynomial time. Step $1$ does \emph{not} extend to general LP, i.e., one cannot efficiently solve an LP given only an efficient separation oracle for its dual~\citep{GLSbook}. Instead, step $1$ crucially exploits the particular ``structure'' of the feasibility constraints defining~\eqref{MOT-P}, details in \S\ref{ssec:alg:step1}. This extends to arbitrary Multimarginal Optimal Transport problems (i.e., arbitrary costs $C$), and therefore may be of independent interest. However, step $1$ is of course useless unless one can efficiently solve the separation oracle in step $2$. Indeed, as mentioned above, step $2$ does \emph{not} extend to general Multimarginal Optimal Transport problems, i.e., there does not exist an efficient implementation of the the separation oracle for~\eqref{MOT-D} for arbitrary costs $C \in \Rntk$. It is here---in step $2$, not step $1$---that we crucially exploit the remaining ``structure'' in the LP reformulation~\eqref{MOT-P} of the barycenter problem, namely the cost $C$ defined in~\eqref{C-WASS}. Intuitively, this ``structure'' of $C$ is geometric: the $n^k$ entries in $C$ correspond to $n^k$ candidate points for the barycenter's support, and these points must lie in certain constrained geometric configurations, details in \S\ref{ssec:alg:step2}.

\par Let us now elaborate on steps $1$ and $2$. To do this, we first recast the separation oracle for~\eqref{MOT-D} in a convenient way for our algorithmic development.

\begin{defin}
	Given $p = (p_1,\ldots,p_k) \in \R^{n \times k}$, the oracle $\SEP$ returns a tuple $\SEP(p) \in \argmin_{\jvec \in [n]^k} C_{\jvec} - \sum_{i=1}^k [p_i]_{j_i}.$
\end{defin}

The intuition behind $\SEP$ is that it implements\footnote{In fact, it can be shown that $\SEP$ is polynomial-time \emph{equivalent} to the separation oracle for~\eqref{MOT-D}, i.e., each oracle can be implemented using polynomial many calls to the other oracle and polynomial additional computation time. However, this is not needed in the sequel.} the separation oracle for~\eqref{MOT-D} because given the tuple $\jvec = \SEP(p)$, exactly one of the following two alternatives must hold: 
\begin{itemize}
	\item $C_{\jvec} - \sum_{i=1}^k [p_i]_{j_i} \geq 0$, in which case this certifies that $p$ is feasible for~\eqref{MOT-D}.
	\item $C_{\jvec} - \sum_{i=1}^k [p_i]_{j_i} < 0$, in which case this provides a hyperplane that separates $p$ from the feasible set of~\eqref{MOT-D}. 
\end{itemize}

\par Now, in terms of this oracle $\SEP$, steps 1 and 2 are formally summarized as follows. 

\begin{prop}[Step 1]\label{prop:generalmot}
	Let $C \in \Rntk$ be an arbitrary cost, and let $\log U$ be the maximum number of bits of precision in an entry of $C$. A vertex solution $P^*$ for~\eqref{MOT-P} can be found in $\poly(n,k,\log U)$ time and $\poly(n,k,\log U)$ calls to $\SEP$. 
\end{prop}

\begin{prop}[Step 2]\label{prop:geomoracle}
	If the cost $C$ is given by \eqref{C-WASS}, the oracle $\SEP(p)$ can be implemented in $\poly(n,k,\log U)$ time, where $\log U$ is the number of bits of precision needed to represent the points $x_{i,j} \in \RR^d$, weights $\lambda_i \in \RR_{> 0}^k$ and potentials $p \in \R^{n \times k}$.
\end{prop}

The remainder of the section is organized as follows. In \S\ref{ssec:alg:step1} and \S\ref{ssec:alg:step2}, we detail the algorithms in steps $1$ and $2$, respectively, and prove Propositions~\ref{prop:generalmot} and~\ref{prop:geomoracle}. Combining Propositions~\ref{prop:generalmot} and~\ref{prop:geomoracle} then proves Theorem \ref{thm:mainbary} aside from checking bit-complexity details, which is done formally in \S\ref{ssec:final}.

\subsection{Step 1: reduction to separation oracle}\label{ssec:alg:step1}

Here we prove Proposition~\ref{prop:generalmot} and describe the algorithm in it. This algorithm has two steps. First, it identifies a small set $S \subseteq [n]^k$ on which an optimal MOT solution is supported. Second, it optimizes over distributions supported on $S$.

\begin{proof}[Proof of Proposition~\ref{prop:generalmot}]
	First, we construct a set $S \subset [n]^k$ as follows. Since $\SEP(p)$ implements a separation oracle for \eqref{MOT-D}, and since~\eqref{MOT-D} has $N = nk$ variables, Theorem~\ref{thm:generalellipguarantee} implies that an optimal solution for \eqref{MOT-D} can be found by the Ellipsoid algorithm in $\poly(N,\log U) = \poly(n,k,\log U)$ time. Let $L$ denote the number of $\SEP$ queries made by the Ellipsoid algorithm. For $l \in [L]$, let $p^{(l)} \in \R^{n \times k}$ be the argument of the $l$-th query to $\SEP$, and let $\jvec^{(l)} \in [n]^k$ be the returned tuple. Let $S := \{\jvec^{(\ell)}\}_{\ell=1}^L$ denote the set of all returned tuples. 
	\par Next, we show how to compute an optimal vertex solution for~\eqref{MOT-P} using $S$. Let $\cM_S(\mu_1, \dots, \mu_k) = \{P \in \Coup : P_{\jvec} = 0\; \forall \jvec \notin S \}$ be the set of distributions in the transportation polytope supported on $S$, and let $\MOTS$ be~\eqref{MOT-P} with the decision set $\Coup$ replaced by $\cM_S(\mu_1, \dots, \mu_k)$. The key lemma is that it suffices to compute an optimal vertex solution for $\MOTS$. This is proved below in Lemma~\ref{lem:step-1}; let us presently show how to use it complete the proof of the main proposition. Note that since the number of Ellipsoid iterations is $\poly(n,k,\log U)$ by Theorem~\ref{thm:generalellipguarantee}, the set $S$ has cardinality of size $\poly(n,k,\log U)$. Thus $\MOTS$ is a polynomial size LP, so we can compute a vertex solution for it in polynomial time with standard LP solvers (e.g., Theorem~\ref{thm:generalellipguarantee}).
\end{proof}

We now state the key lemma used in the above proof.
\begin{lemma}\label{lem:step-1}
	Any vertex solution $P$ for $\MOTS$ is also a vertex solution for~\eqref{MOT-P}. 
\end{lemma}
Lemma~\ref{lem:step-1} follows directly from the following three observations. Below, let $C'$ denote the tensor that agrees with $C$ on $S$, and equals $2U$ elsewhere. Also let $\MOTp$ denote the problem~\eqref{MOT-P} where the cost $C$ is replaced by $C'$. 

\begin{obs}\label{obs:C':val}
	The optimal values of~\eqref{MOT-P} and $\MOTp$ are equal.
\end{obs}
\begin{proof}
	Since the Ellipsoid algorithm is deterministic and accesses the cost only through the $\SEP$ oracle, an inductive argument shows that all for all $l \in [L]$,
	$$\min_{\jvec \in [n]^k} C'_{\jvec} - \sum_{i=1}^k [p_i^{(l)}]_{j_i} = \min_{\jvec \in [n]^k} C_{\jvec} - \sum_{i=1}^k [p_i^{(l)}]_{j_i}.$$ 
	That is, $C'$ is consistent with $C$ on the $\SEP$ queries made by the Ellipsoid algorithm. Therefore~\eqref{MOT-D} has the same value with cost $C'$ or $C$. We conclude by strong duality.
\end{proof}

\begin{obs}\label{obs:motsmotpval}
	The optimal values of $\MOTp$ and $\MOTS$ are equal.
\end{obs}
\begin{proof}
	Since $C$ and $C'$ agree on $S$, it suffices to show that every optimal solution for $\MOTp$ is supported on $S$. Suppose for contradiction that there exists an optimal solution $P$ for $\MOTp$ that is not supported on $S$. Then we must have $\langle P, C \rangle < \langle P, C' \rangle$, since $C'$ is strictly larger than $C$ on $[n]^k \sm S$. However, since $P$ is feasible for \eqref{MOT-P}, this implies that the value of \eqref{MOT-P} is strictly less than that of $\MOTp$, contradicting Observation~\ref{obs:C':val}. 
\end{proof}

\begin{obs}\label{obs:vertex}
	Every vertex of $\cM_S(\mu_1, \dots, \mu_k)$ is a vertex of $\Coup$.
\end{obs}
\begin{proof}
	Let $P$ be a vertex of $\cM_S(\mu_1, \dots, \mu_k)$, and let $
	P = \lambda Q + (1-\lambda) R
	$
	for $\lambda \in [0,1]$ and $Q,R \in \Coup$. It suffices to show that $Q = R = P$. Since $P$ is supported on tuples in $S$, and since $Q$ and $R$ are entrywise non-negative, we have that $Q$ and $R$ are also supported on $S$. Therefore $Q, R \in \cM_S(\mu_1, \dots, \mu_k)$. But since $P$ is a vertex of $\cM_S(\mu_1, \dots, \mu_k)$, this implies that $Q = R = P$.
\end{proof}

\subsection{Step 2: efficient algorithm for the separation oracle}\label{ssec:alg:step2}

Here we prove Proposition~\ref{prop:geomoracle} and describe the algorithm in it for efficiently implementing the $\SEP$ oracle for the cost $C$ in~\eqref{C-WASS}. Recall that this $\SEP$ oracle requires computing an optimal tuple for 
\begin{equation}\argmin_{\jvec \in [n]^k} \min_{y \in \RR^d} g(\jvec,y) \label{sep-oracle}
\end{equation}
where
\[
g(\jvec,y) := \sum_{i=1}^k \lambda_i (\|x_{i,j_i} - y\|^2 - [w_i]_{j_i}),
\]
and $w_i$ denotes $p_i/\lambda_i$. At a high level, our approach is to swap the order of minimization, optimize over $y \in \R^d$, and then (easily) recover an optimal tuple from this optimal $y$. The difficulty is in the optimization over $y \in \R^d$. The key to performing this efficiently is partitioning the space $\R^d$ into a ``cell complex'' such that (i) the optimization over $y$ in each cell is easy, and (ii) there are only polynomially many cells. Operationally, this allows us to reduce the separation oracle optimization~\eqref{sep-oracle} to optimizing over only a polynomially sized set of candidate tuples in $[n]^k$---one for each cell---which we moreover show can be efficiently identified and enumerated.

\par To formalize this, we make the following key definitions.
Define for $i \in [k]$ and $j \in [n]$ the set
\begin{align}
	E_{i,j} =
	\{y \in \RR^d \, : \, \|x_{i,j} - y\|^2 - [w_i]_j < \|x_{i,j'} - y\|^2 - [w_i]_{j'}, \, \forall j' \neq j \},
	\label{eq:Eij}
\end{align}
and define for each tuple $\jvec \in [n]^k$ the set
\begin{align}
	F_{\jvec} = \bigcap_{i=1}^k E_{i,j_i}.
	\label{eq:cF}
\end{align}
Geometrically, for each $i \in [k]$, the cells $\{E_{i,j}\}_{j \in [n]}$ form a power diagram (see \S\ref{ssec:prelim:cg}) on the spheres $S(x_{i,1},r_{i,1}),\ldots,S(x_{i,n},r_{i,n})$, where the $j$-th sphere is centered at point $x_{i,j}$ and has radius $r_{i,j} := \sqrt{[w_i]_j - \min_{j'} [w_i]_{j'}} \geq 0$. Each power diagram ``essentially'' partitions $\R^d$ in the sense that its constituent cells are disjoint and their closures cover $\R^d$; see Figure~\ref{fig:powerdiagexample} for an illustration. The cell complex $\{F_{\jvec}\}_{\jvec \in [n]^k}$ is the intersection of these $k$ power diagrams and ``essentially'' partitions $\R^d$ in the analogous way; see Figure~\ref{fig:powerdiagoverlay} for an illustration.

\begin{figure}
	\centering
	\begin{tabular}{@{}c@{}c@{}ccc@{}}
		\begin{tabular}{@{}c@{}}
			\includegraphics[clip, trim = 3.6cm 1.2cm 3.2cm 1.2cm, scale=0.37]{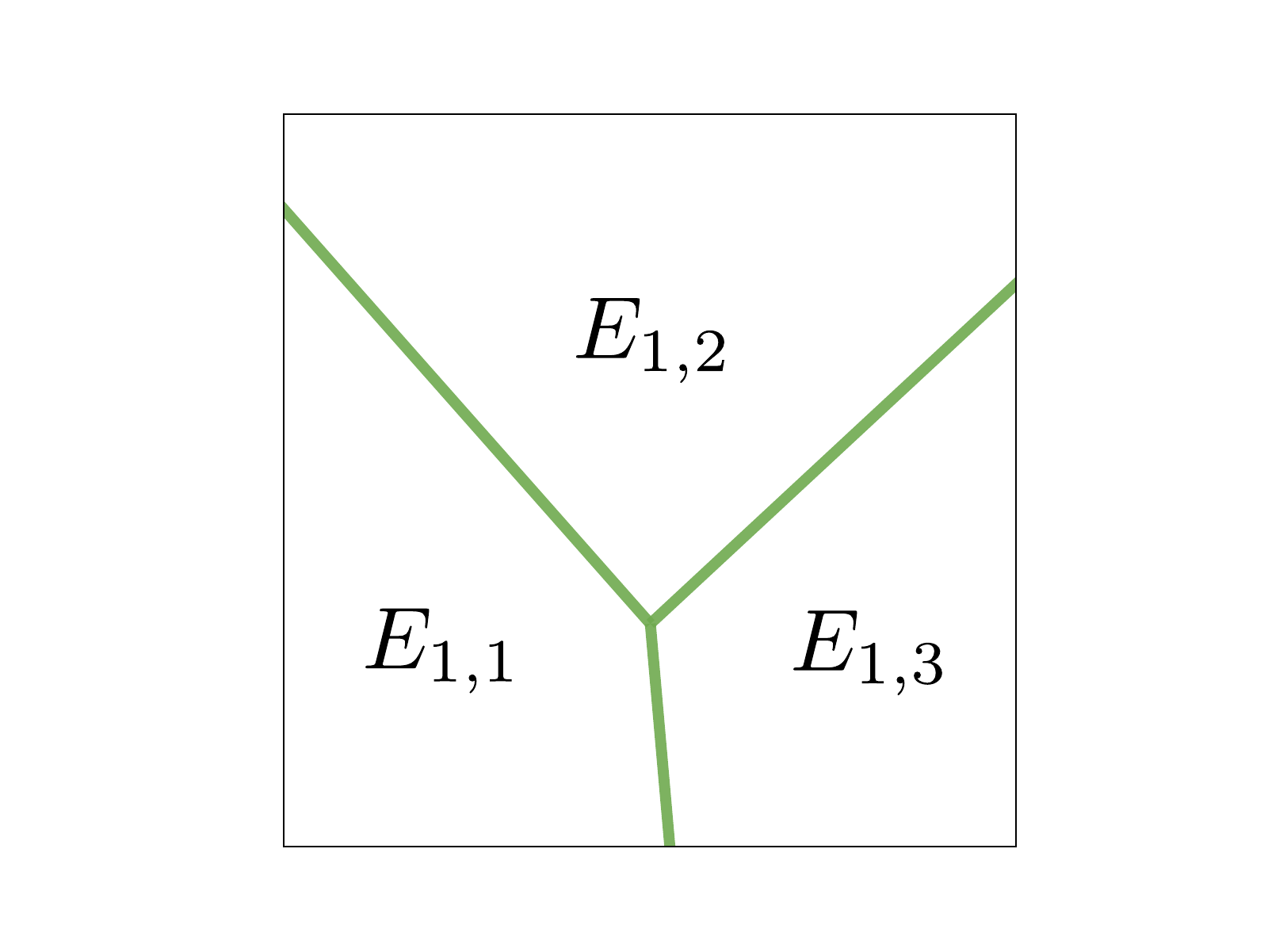}
		\end{tabular} \ 
		&
		\begin{tabular}{@{}c@{}}
			\includegraphics[clip, trim = 3.6cm 1.2cm 3.2cm 1.2cm, scale=0.37]{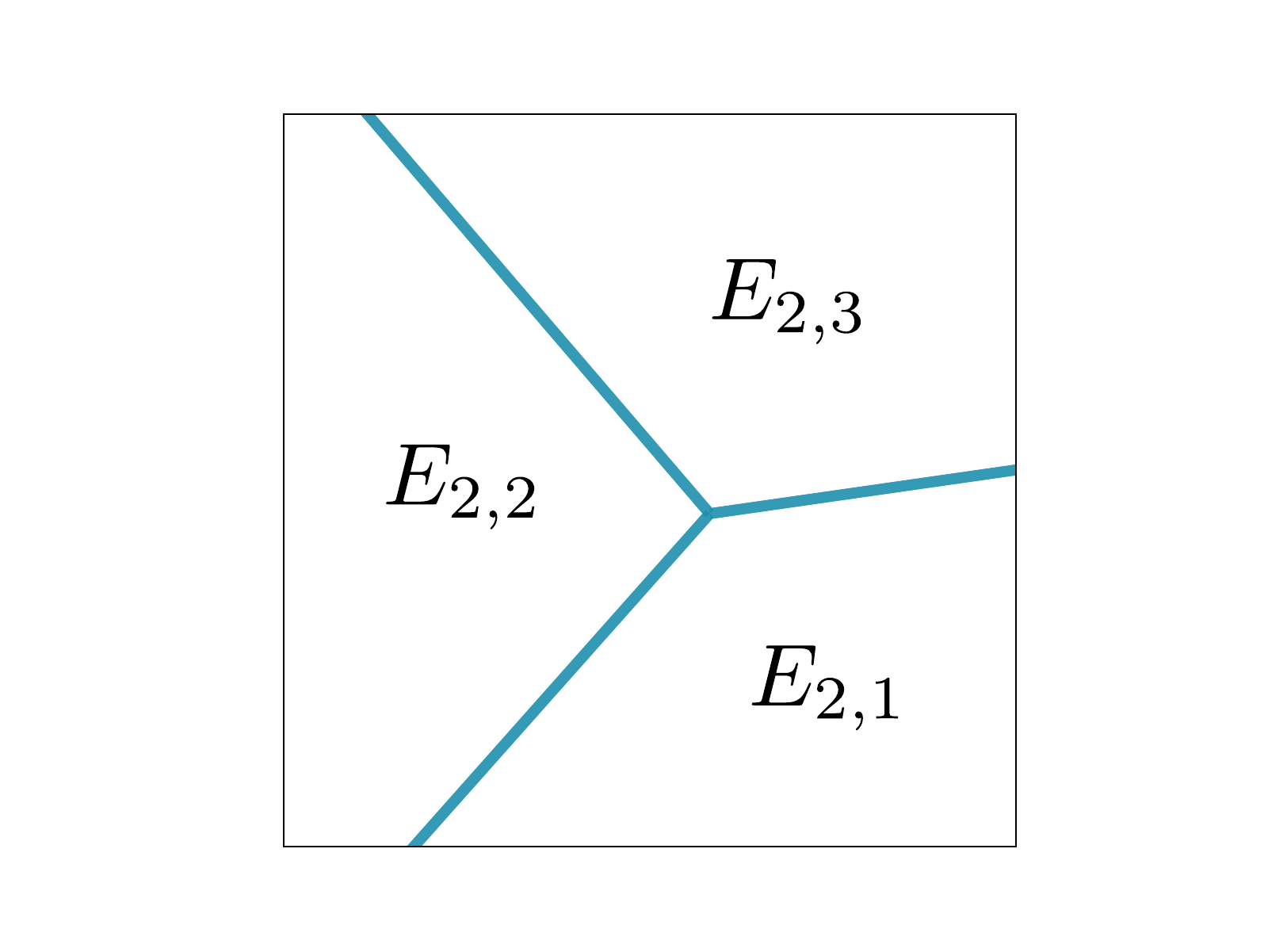}
		\end{tabular} \ 
		&
		\begin{tabular}{@{}c@{}}
			\includegraphics[clip, trim = 3.6cm 1.2cm 3.2cm 1.2cm, scale=0.37]{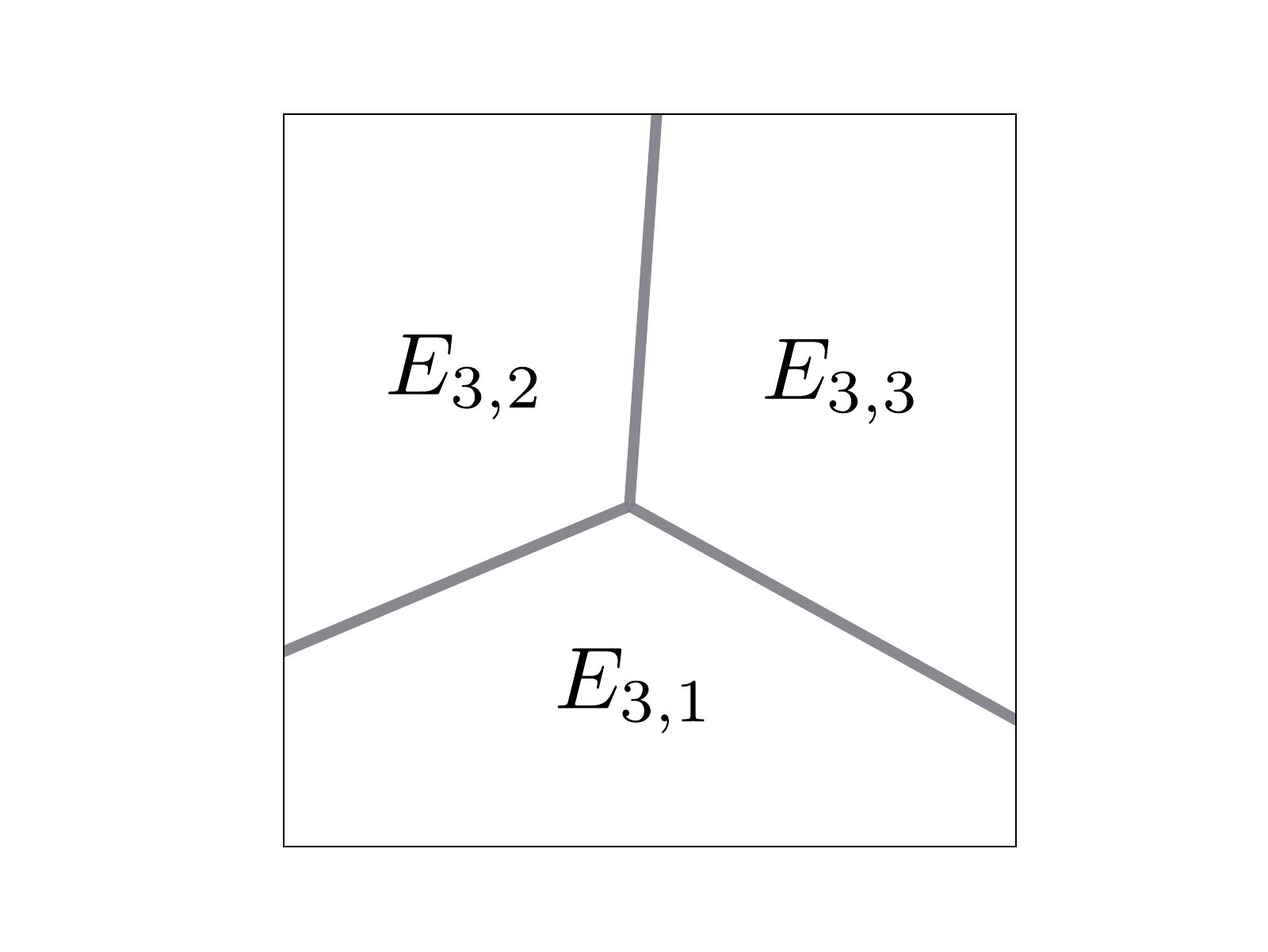}
		\end{tabular}
		&
		\begin{tabular}{@{}c@{}}
			$\mathbf{\xrightarrow{\text{\textbf{intersect}}}}$
		\end{tabular}
		&
		\begin{tabular}{@{}c@{}}
			\includegraphics[clip, trim = 3.6cm 1.2cm 3.2cm 1.2cm, scale=0.37]{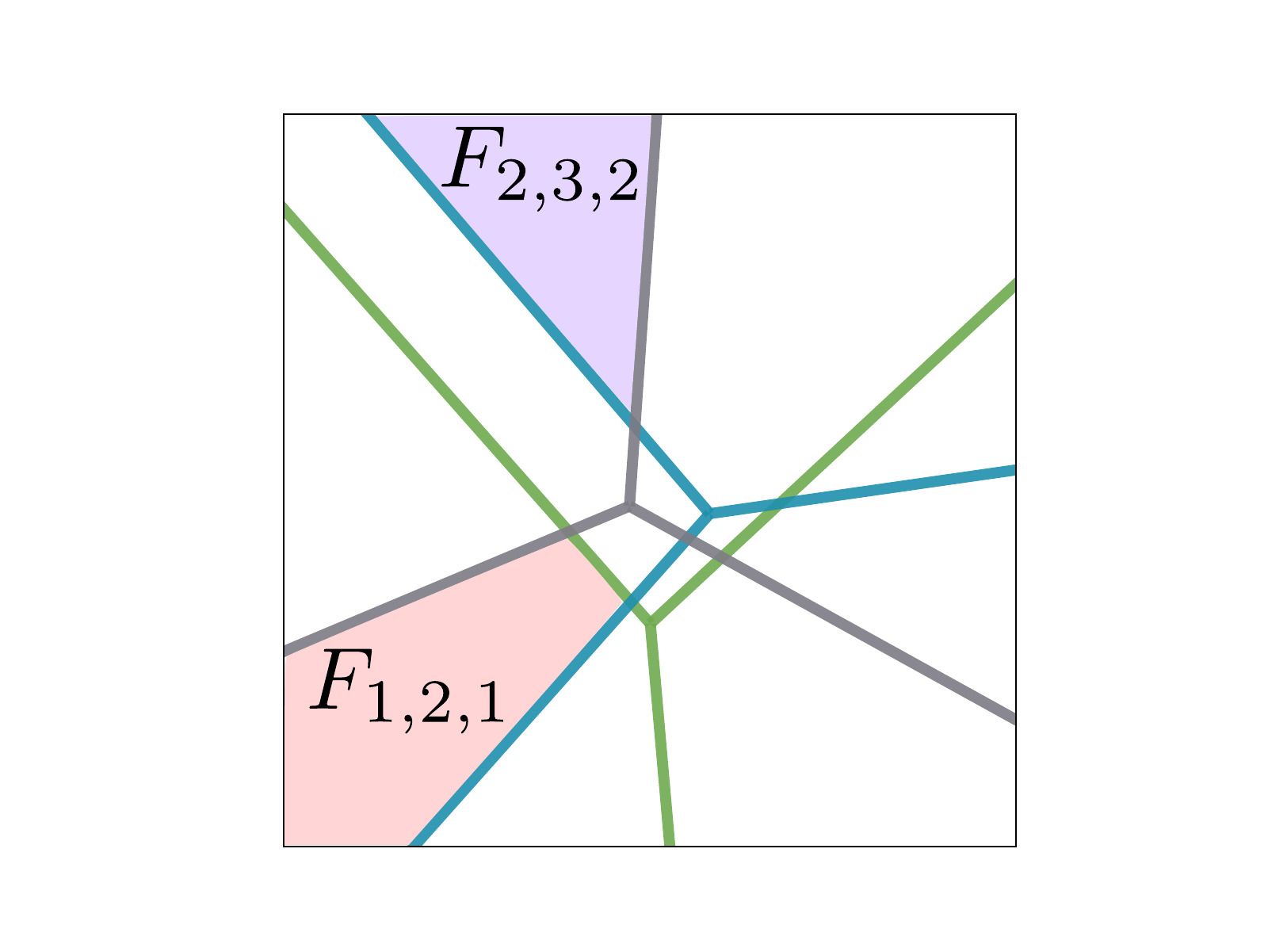}
		\end{tabular}
	\end{tabular}
	\caption{Illustrates $k = 3$ power diagrams $\{ \{E_{i,j}\}_{j \in [n]} \}_{i \in [k]}$ each with $n=3$ cells, and their intersection $\{F_{\jvec}\}_{\jvec \in [n]^k}$. For instance, the red cell in the intersected diagram is $F_{1,2,1}= E_{1,1} \cap E_{2,2} \cap E_{3,1}$, and the purple cell is $F_{2,3,2} = E_{1,2} \cap E_{2,3} \cap E_{3,2}$. Note that the intersected diagram has only $13$ non-empty cells, which is less than $n^k = 27$ (c.f., Lemma~\ref{lem:sep-runtime}).} \label{fig:powerdiagoverlay}
\end{figure}

The heart of our algorithm lies in the following two lemmas. The first lemma shows that the optimization~\eqref{sep-oracle} over the exponentially many tuples $j \in [n]^k$ may be restricted to just those whose corresponding cell $F_{\jvec}$ is non-empty, i.e., we may restrict to the tuples in
\begin{align}
	T := \{\jvec \in [n]^k \; : \; F_{\jvec} \neq \emptyset \}.
\end{align}
The second lemma shows that this candidate set $T$ contains only polynomially many tuples and moreover can be efficiently enumerated. Briefly, the first lemma exploits the fact that the optimization over $y \in \R^d$ is equivalent to optimizing over the cells in $F_{\jvec}$, and the second lemma exploits complexity bounds for the intersections of power diagrams. Together, these lemmas are sufficient to efficiently solve the separation oracle because for any fixed $\jvec$, the value $\min_{y \in \R^d} g(\jvec,y)$ can be efficiently computed in closed-form (as shown below in~\eqref{sep-oracle:T}).

\begin{lemma}\label{lem:sep-correct}
	The optimization over $\jvec \in [n]^k$ in the separation oracle problem~\eqref{sep-oracle} can be equivalently restricted to $\jvec \in T$. That is,
	\[
	\min_{\jvec \in [n]^k} \min_{y \in \R^d} g(\jvec,y)
	=
	\min_{\jvec \in T} \min_{y \in \R^d} g(\jvec,y).
	\]
\end{lemma}

\begin{proof}
	The inequality ``$\leq$'' is obvious; we show the other inequality ``$\geq$''. By swapping the order of minimization and using the fact that $\{F_{\jvec}\}_{\jvec \in T}$ cover $\R^d$ modulo closure, we have
	\begin{align*}
		\min_{\lvec \in [n]^k}	\min_{y \in \R^d} g(\lvec,y)
		=
		\min_{y \in \R^d} \min_{\lvec \in [n]^k}  g(\lvec,y)
		=
		\min_{\jvec \in T} \min_{y \in \overline{F_{\jvec}}} \min_{\lvec \in [n]^k}  g(\lvec,y).
	\end{align*}
	We claim that the inner minimization over $\lvec$ is explicit: $\lvec = \jvec$. Indeed, by separability of $g$ in the coordinates of $\lvec$ and non-negativity of $\lambda_i$, for each $i \in [n]$ the optimal $\ell_i$ is a solution of $\argmin_{\ell_i \in [n]} \|x_{i,\ell_i} - y\|^2 - [w_i]_{\ell_i}$; and $j_i$ is a solution of this by definition of $E_{i,j_i}$ (see~\eqref{eq:Eij}) and the fact that $\overline{E_{i,j_i}}$ contains $y$ (by definition of $F_{\jvec}$, see~\eqref{eq:cF}). Therefore
	\begin{align*}
		\min_{\jvec \in T} \min_{y \in \overline{F_{\jvec}}} \min_{\lvec \in [n]^k}  g(\lvec,y)
		=
		\min_{\jvec \in T} \min_{y \in \overline{F_{\jvec}}}
		g(\jvec,y).
	\end{align*}
	Now by enlarging the optimization region, we have the simple bound
	\begin{align*}
		\min_{\jvec \in T} \min_{y \in \overline{F_{\jvec}}}
		g(\jvec,y)
		\geq 
		\min_{\jvec \in T} \min_{y \in \R^d}
		g(\jvec,y).
	\end{align*}
	Combining the above three displays completes the proof.
\end{proof}

\begin{lemma}\label{lem:sep-runtime}
	For any fixed dimension $d$, the set $T$ can be enumerated in $\poly(n,k,\log U)$ time.
\end{lemma}
\begin{proof}
	By Lemma \ref{lem:powdiagramruntime}, the $O(nk)$ total facets for the $k$ power diagrams $\{\{E_{i,j}\}_{i \in [n]}\}_{j \in [k]}$ can be computed in $\poly(n,k,\log U)$ time. For each facet, compute the $(d-1)$-dimensional hyperplane it lies in. The cell complex $\cH$ formed by these hyperplanes is a subcomplex of the cell complex formed by intersecting the power diagrams. By Lemma~\ref{lem:hyperplaneintersectionenumeration}, we can enumerate the cells in $\cH$ in $\poly(n,k,\log U)$ time. For each cell in $\cH$, the corresponding tuple $\jvec \in [n]^k$ is computable in $O(nk \cdot \polylog \; U)$ time by computing the $k$ coordinates of the tuple separately. Since each non-empty cell $F_{\jvec}$ contains at least one cell in $\cH$, this process enumerates all tuples in $T$.
\end{proof}

We now conclude the desired efficient algorithm for the separation oracle. 
	\begin{proof}[Proof of Proposition~\ref{prop:geomoracle}]
	By Lemma~\ref{lem:sep-correct}, it suffices to output a tuple $\jvec$ minimizing $\min_{\jvec \in T} \min_{y \in \R^d} g(\jvec,y)$. Since $\sum_{i=1}^k \lambda_i x_{i,j_i} \in \argmin_{y \in \R^d} g(\jvec,y)$, it therefore suffices to solve
	\begin{align}
		\argmin_{\jvec \in T}
		\sum_{i=1}^k \lambda_i \|x_{i,j_i}\|^2 
		- \|\sum_{i=1}^k \lambda_i x_{i,j_i}\|^2
		- \sum_{i=1}^k \lambda_i [w_i]_{j_i}.
		\label{sep-oracle:T}
	\end{align}
	Peform this by enumerating the set $T$ using the algorithm in Lemma~\ref{lem:sep-runtime}.
	\end{proof}

\subsection{Putting the pieces together}\label{ssec:final}

\begin{proof}[Proof of Theorem~\ref{thm:exactbary}]
	Assume each $x_{i,j_i}$ and $\lambda_i$ is written to $\log U$ bits of precision.
	Since each entry of the cost tensor \eqref{C-WASS} requires only $O(\log k + \log U)$ bits of precision, and since the parameter $w \in \R^{n \times k}$ in each $\SEP$ query made by the algorithm in Proposition~\ref{prop:generalmot} requires only $\poly(n,k,\log U)$ bits of precision, it follows that the algorithm in Proposition~\ref{prop:generalmot} combined with the $\SEP$ oracle implementation in Proposition~\ref{prop:geomoracle} computes a vertex solution $P^*$  for \eqref{MOT-P} in $\poly(n,k,\log U)$
	time. By Lemma~\ref{lem:vertex-sparse}, $P^*$ has at most $nk-k+1$ non-zero entries. Thus we can recover from $P^*$ an optimal barycenter $\nu$ with support size at most $nk-k+1$ in time $\poly(n,k,\log U)$ by the reduction in Lemma~\ref{lem:bary-to-mot}.
\end{proof}

	\begin{proof}[Proof of Theorem~\ref{thm:mainbary}]
	By rounding both the weights $\lambda_i$ and the coordinates of the atoms $x_{i,j} \in \R^d$ to $\poly(\eps/(Rkd))$ additive accuracy, it can be ensured that each of these numbers requires only $O(\log (Rkd/\eps))$ bits of precision and also that the objective function $\nu \mapsto \sum_{i=1}^k \lambda_i \cW(\mu_i,\nu)$ for the barycenter optimization~\eqref{eq:intro:bary} is preserved pointwise to $\eps$ additive accuracy. This follows from a straightforward calculation and the fact (immediate from the definition of optimal transportation~\citep[\S1]{Vil03} and an application of H\"older's inequality) that if the squared Euclidean distance between each atom of $\mu_i$ and each atom of $\nu$ is preserved up to $\eps'$ additive accuracy, then the squared $2$-Wasserstein distance $\cW(\mu_i,\nu)$ is preserved up to $\eps'$ additive accuracy. Now solve the barycenter problem for the rounded weights and atoms exactly using Theorem~\ref{thm:exactbary}.
\end{proof}

\section{Numerical implementation}\label{sec:experiments}

While the focus of this paper is theoretical, here we briefly mention that a slight variant of our algorithm can provide high-precision solutions at previously intractable problem sizes. To demonstrate this, we implement our algorithm for dimension $d=2$ in Python. The only difference between our numerical implementation and the theoretical algorithm described above is that we use a standard cutting-plane method (see, e.g.,~\citep[\S6.3]{BerTsi97}) for the ``outer loop'' in step $1$ rather than the Ellipsoid algorithm due to its good practical performance. Code and further implementation details are provided on Github.\footnote{\url{https://github.com/eboix/high_precision_barycenters}}

\subsection{Computing exact solutions at previously intractable scales}

\par Figure~\ref{fig:exp} demonstrates that our algorithm solves the barycenter problem~\eqref{eq:intro:bary} to machine precision on an instance with $k=10$ uniform distributions each on $n=20$ points randomly drawn from $[-1,1]^2 \subset \R^2$. In contrast, existing popular barycenter algorithms which use the fixed-support assumption can converge faster but only to lower-precision approximations. This is because the $\Theta(1/\eps^d)$ gridsize that they require for $\eps$-additive approximation results in a large-scale LP which is prohibitive even for relatively low precision $\eps$; see \S\ref{sec:prev} for details. Note also that a standard LP solver requires optimizing over $n^k = 20^{10} \approx 10^{13}$ variables for the LP formulation~\eqref{MOT-P} and thus is clearly infeasible at this scale.

\begin{figure}[h]
	\centering
	\begin{subfigure}{.45\textwidth}
		\centering
		\includegraphics[clip, trim = 0.5cm 0cm 3cm 1cm, scale=0.28]{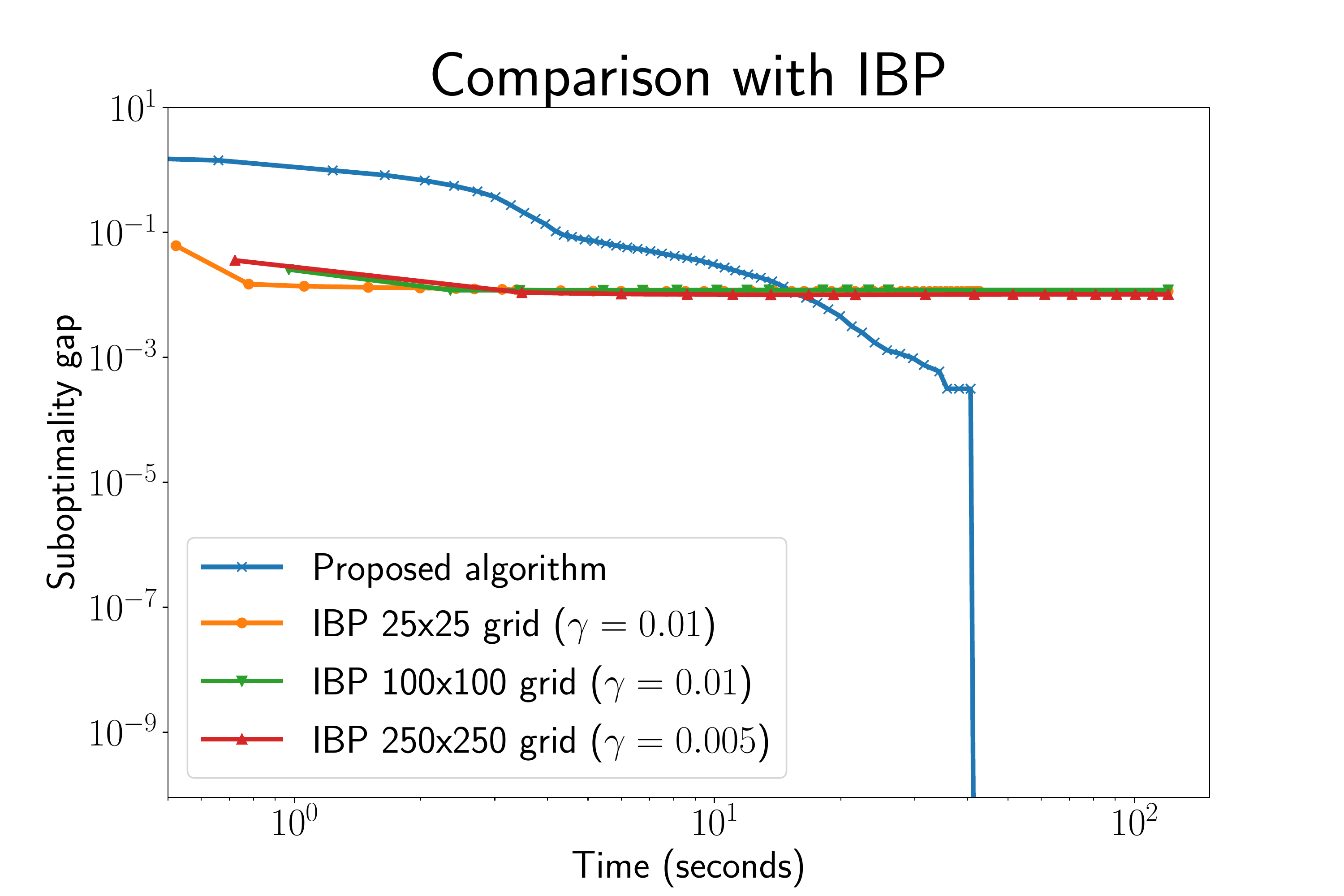}
		\caption{
			Comparison with the Iterated Bregman Projection (IBP) algorithm of~\citep{solomon2015convolutional} using their implementation \url{https://github.com/gpeyre/2015-SIGGRAPH-convolutional-ot}. 
		}
		\label{fig:exp:ibp}
	\end{subfigure}%
	\hfill
	\begin{subfigure}{.45\textwidth}
		\centering
		\includegraphics[clip, trim = 0.5cm 0cm 3cm 1cm, scale=0.28]{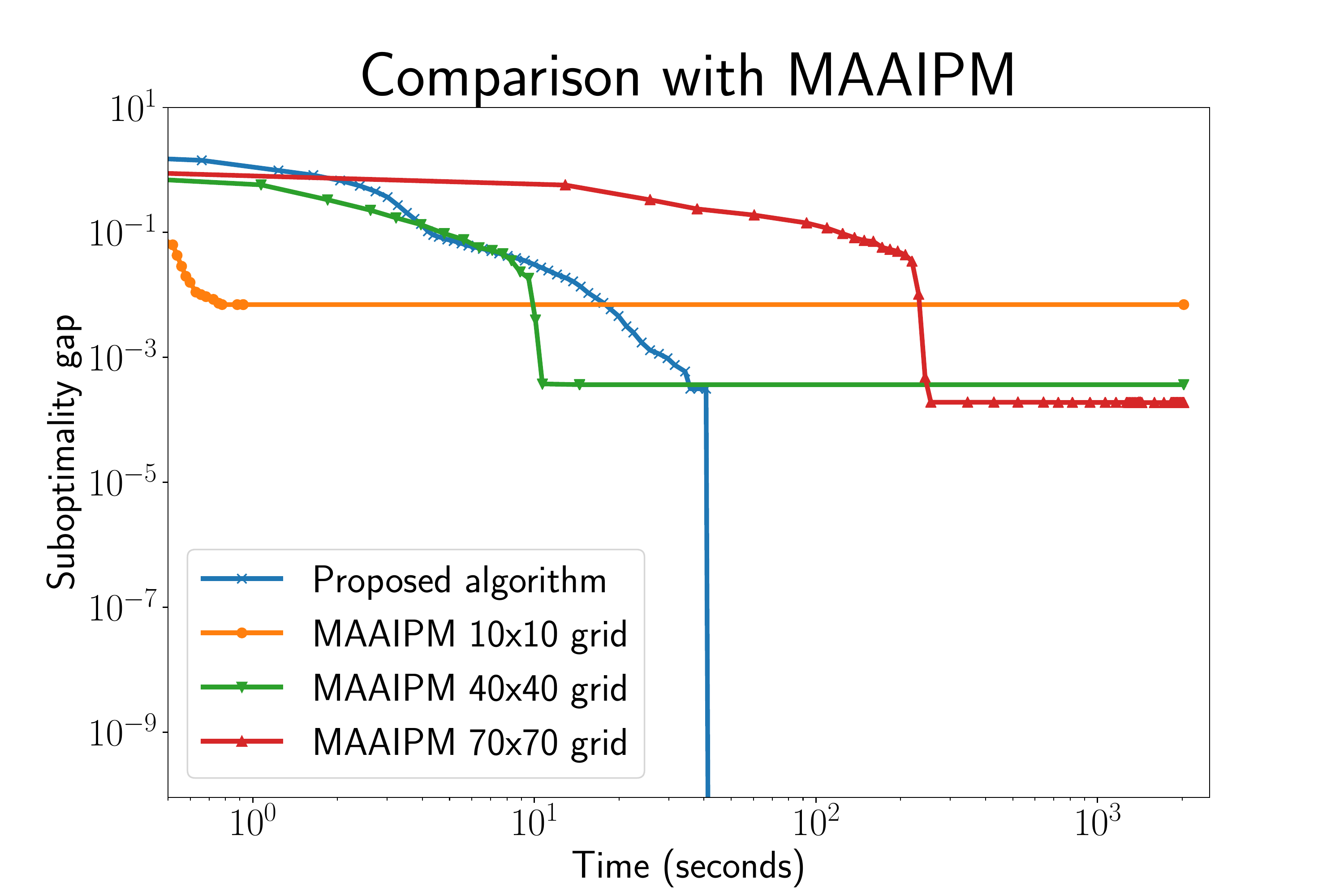}
		\caption{Comparison with the Matrix-based Adaptive Alternating Interior-Point Method (MAAIPM) of~\citep{ge2019interior} using their implementation \url{https://gitlab.com/ZXiong/wasserstein-barycenter}.
		}	\label{fig:exp:ipm}
	\end{subfigure}
	\caption{
		Comparison with state-of-the-art algorithms. The $y$-axis is the suboptimality for the barycenter optimization~\eqref{eq:intro:bary}; note that while standard LP solvers cannot be run at this scale, our algorithm yields an exact solution (certified by our separation oracle) which enables plotting this suboptimality.
		Both compared algorithms require a fixed-support assumption and are run on uniform grids of increasing sizes. IBP has an additional parameter: the entropic regularization $\gamma$, which significantly impacts the algorithm's accuracy and numerical stability, see~\citep{solomon2015convolutional,PeyCut17}. We provide a generous comparison here for IBP by (i) fine-tuning $\gamma$ for it (we binary search for the most accurate $\gamma$; note that their code does not always converge for $\gamma$ small due to numerical instability); and (ii) exactly computing the Wasserstein distances $\cW(\mu_i,\nu)$ to IBP's current barycenter $\nu$ in the barycenter objective~\eqref{eq:intro:bary} using~\citep{Lemon}, which is more accurate than IBP's approximation (this is slow for large grids but is not counted in IBP's timing). Our algorithm finds an exact barycenter after ${\sim}50$ seconds. All experiments are run on a standard $2014$ Lenovo Yoga 720-13IKB laptop.
	}
	\label{fig:exp}
\end{figure}

\subsection{Sharper visualizations}
Here we demonstrate that the high-precision solutions computed by our algorithm yield significantly sharper visualizations than the low-precision solutions that were previously computable. Specifically, here we compare our barycenter algorithm against state-of-the-art methods on a standard benchmark dataset of images of nested ellipses~\citep{CutDou14,janati2020debiased}. This dataset consists of $k = 10$ images, each of size $60 \times 60$. Five of these images are shown in Figure~\ref{fig:ellipses}.

\begin{figure}
	\centering
	\begin{tabular}{ccccc}
		\includegraphics[scale=0.25]{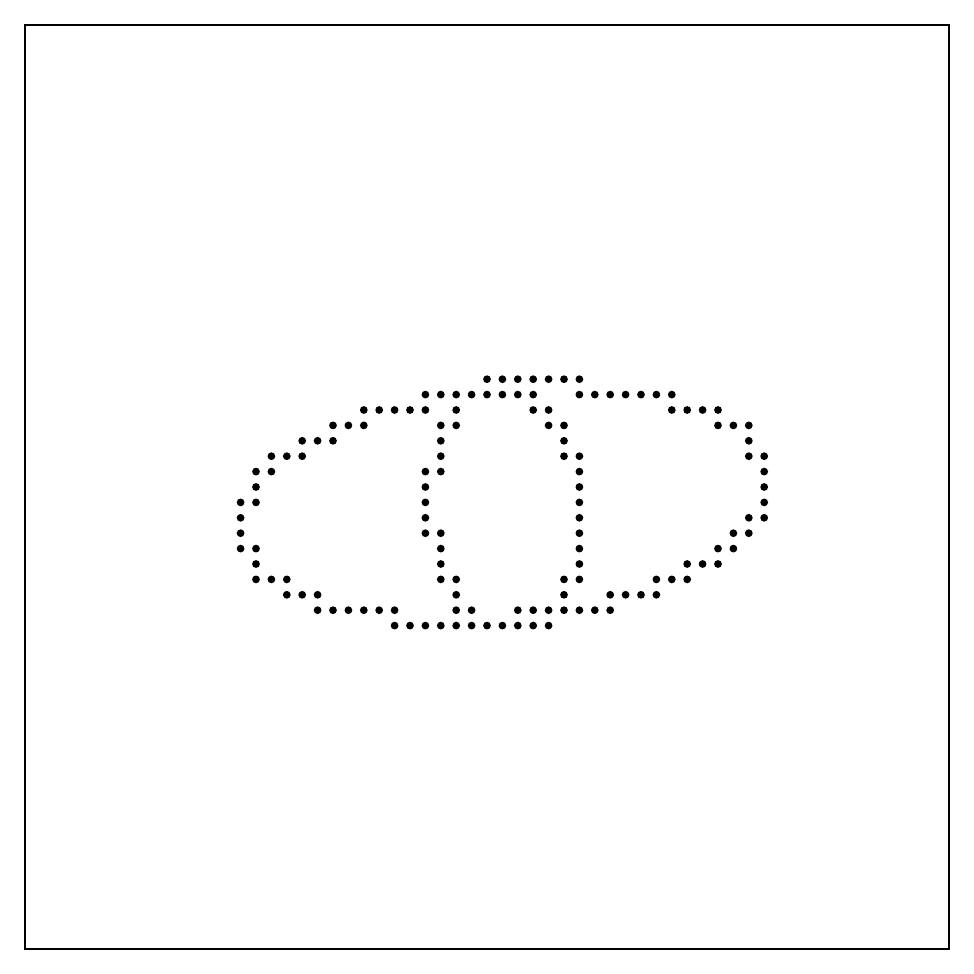} & 
		\includegraphics[scale=0.25]{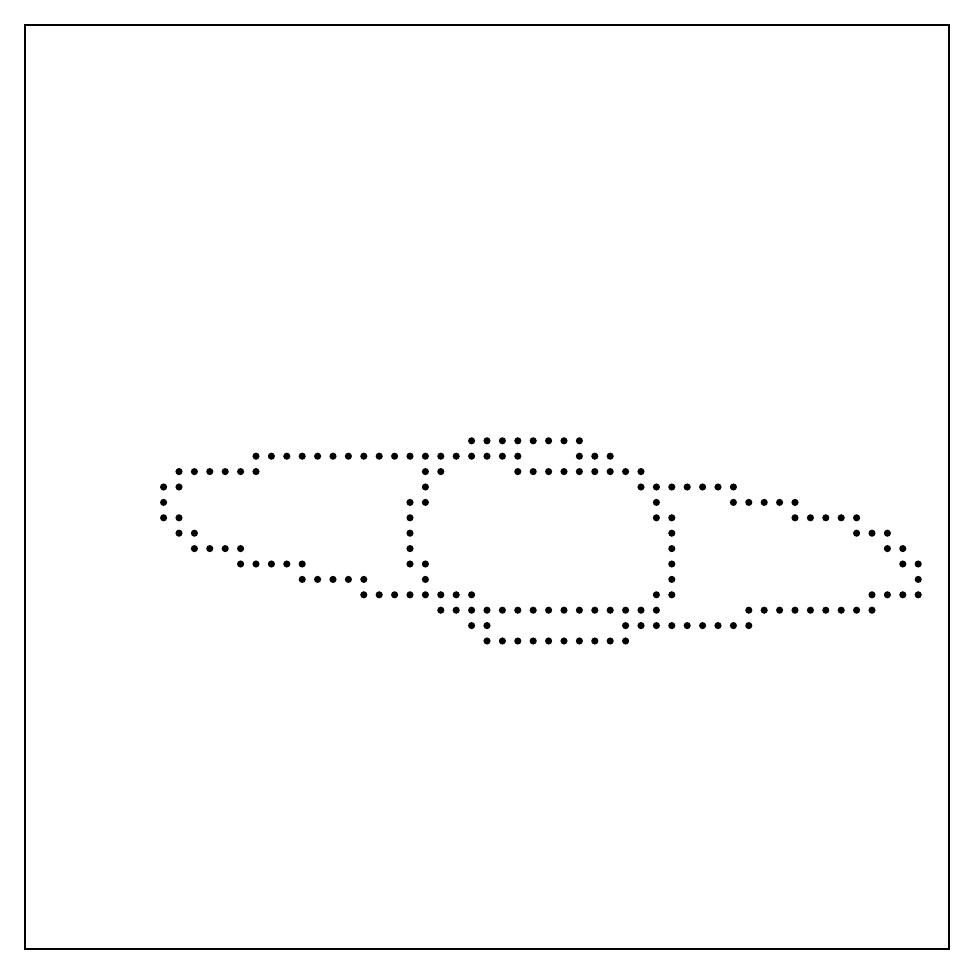} & 
		\includegraphics[scale=0.25]{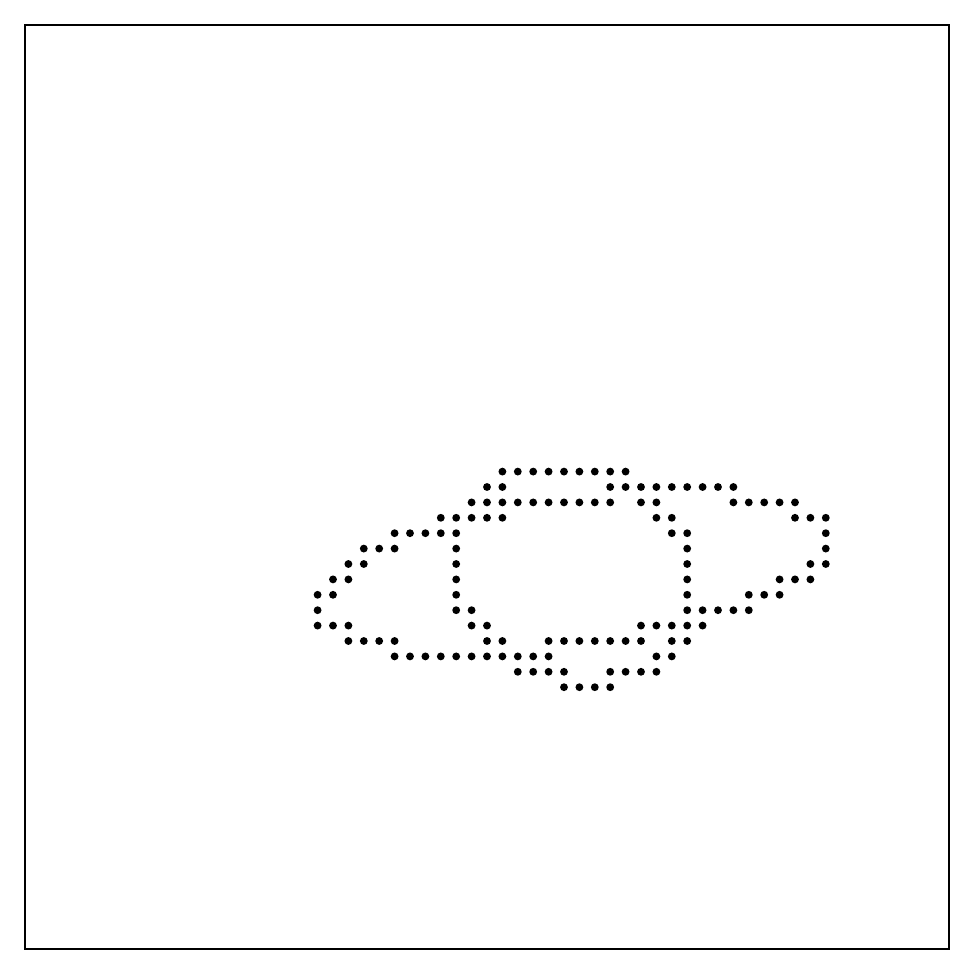} & 
		\includegraphics[scale=0.25]{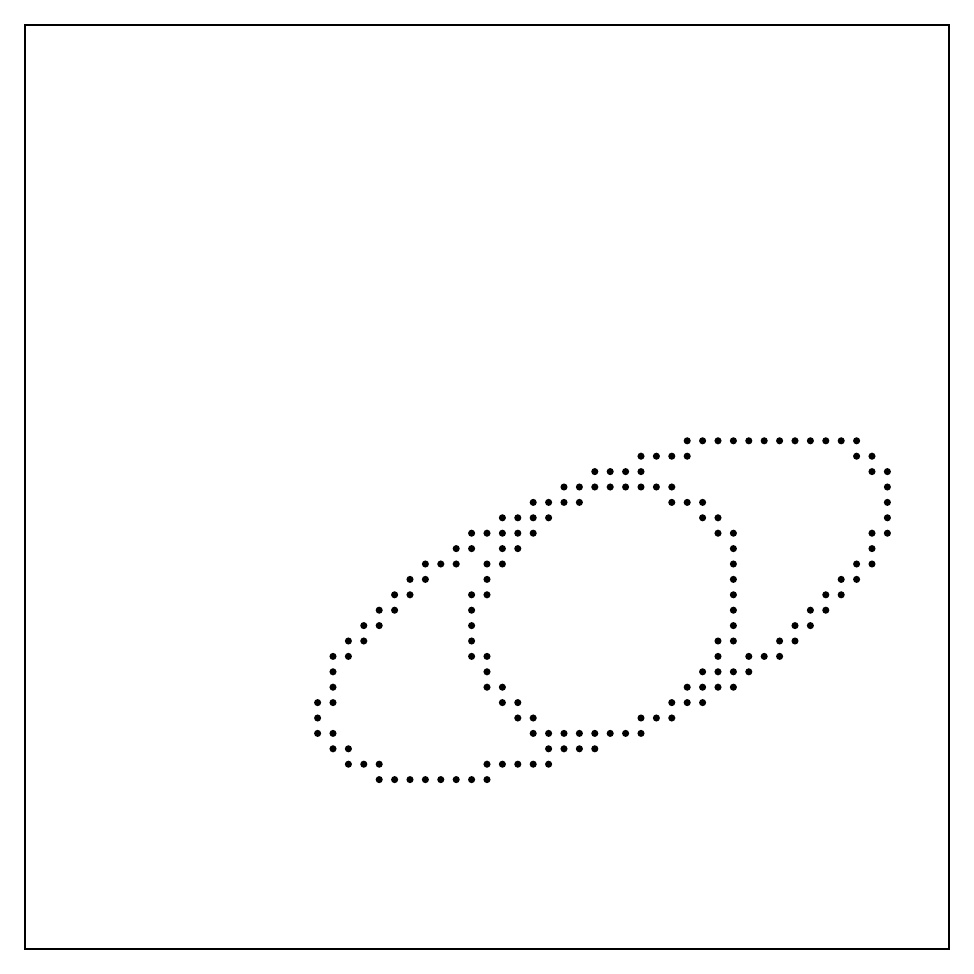} & 
		\includegraphics[scale=0.25]{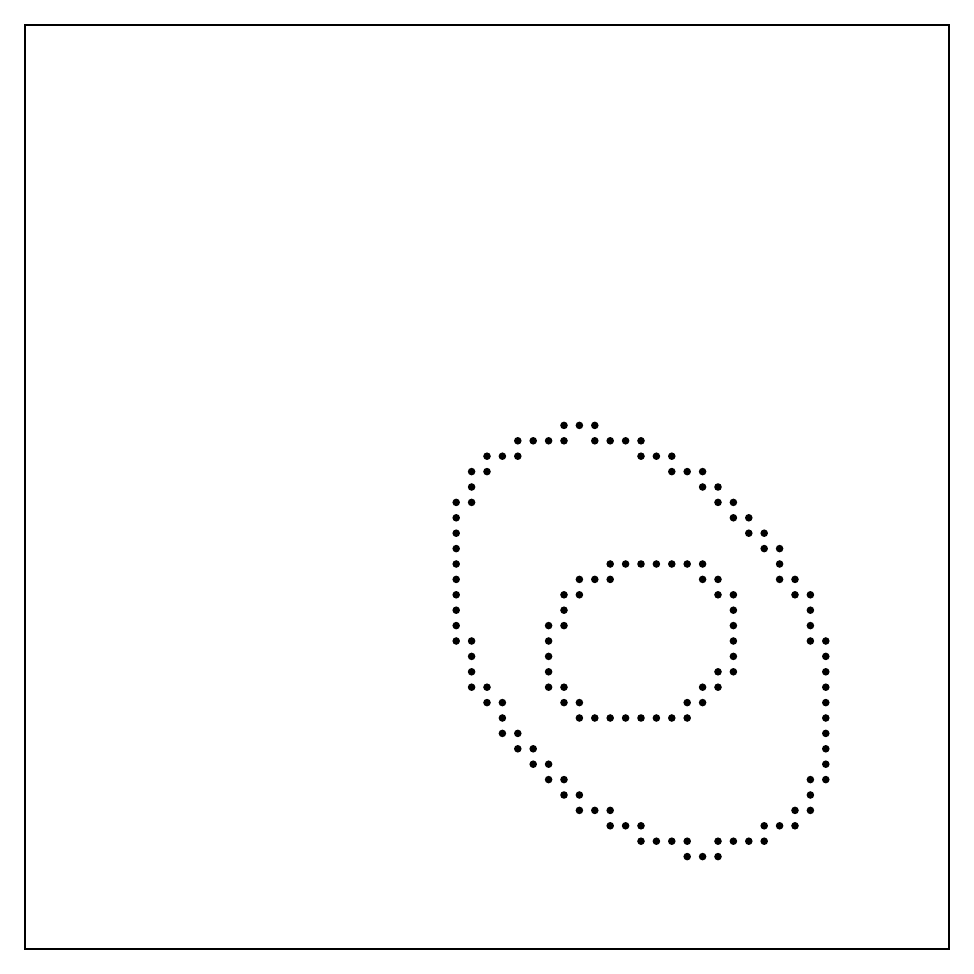}
	\end{tabular}
	\caption{Five sample images from the nested ellipses dataset in~\citep{janati2020debiased}.}\label{fig:ellipses}
\end{figure}

Figure~\ref{fig:ellipsesbarycenters} contains a visual comparison of the exact barycenter computed by our algorithm and the approximate barycenters produced by the most competitive algorithms tested in the recent paper \citep{janati2020debiased}. All are fixed-support algorithms except for the algorithm of~\citep{luise2019sinkhorn}.

\begin{figure}
	\centering
	\scriptsize{
		\begin{tabular}{@{}ccccc@{}}
			Ours &
			MAAIPM \citep{ge2019interior} &
			Debiased \citep{janati2020debiased} &
			IBP \citep{solomon2015convolutional} &
			Frank-Wolfe \citep{luise2019sinkhorn} 
			\\
			\begin{tabular}{@{}c@{}}\includegraphics[scale=0.25]{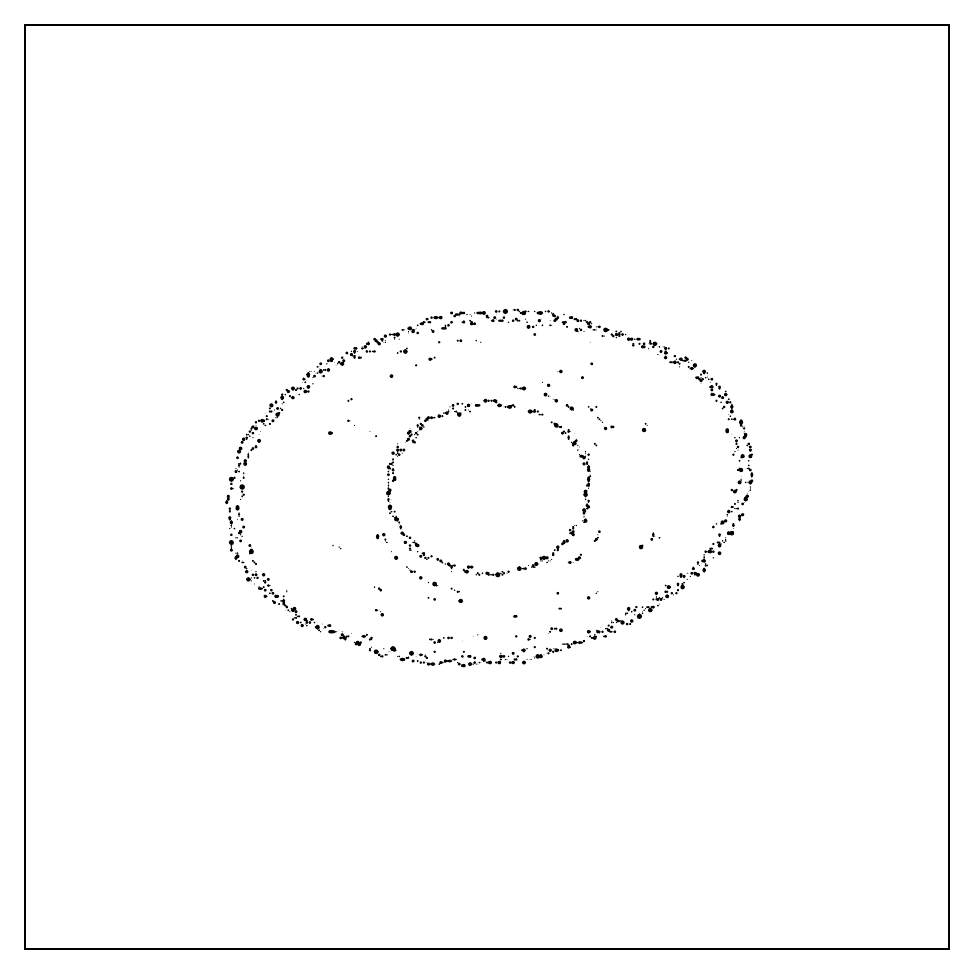}
			\end{tabular} &
			\begin{tabular}{@{}c@{}}\includegraphics[scale=0.25]{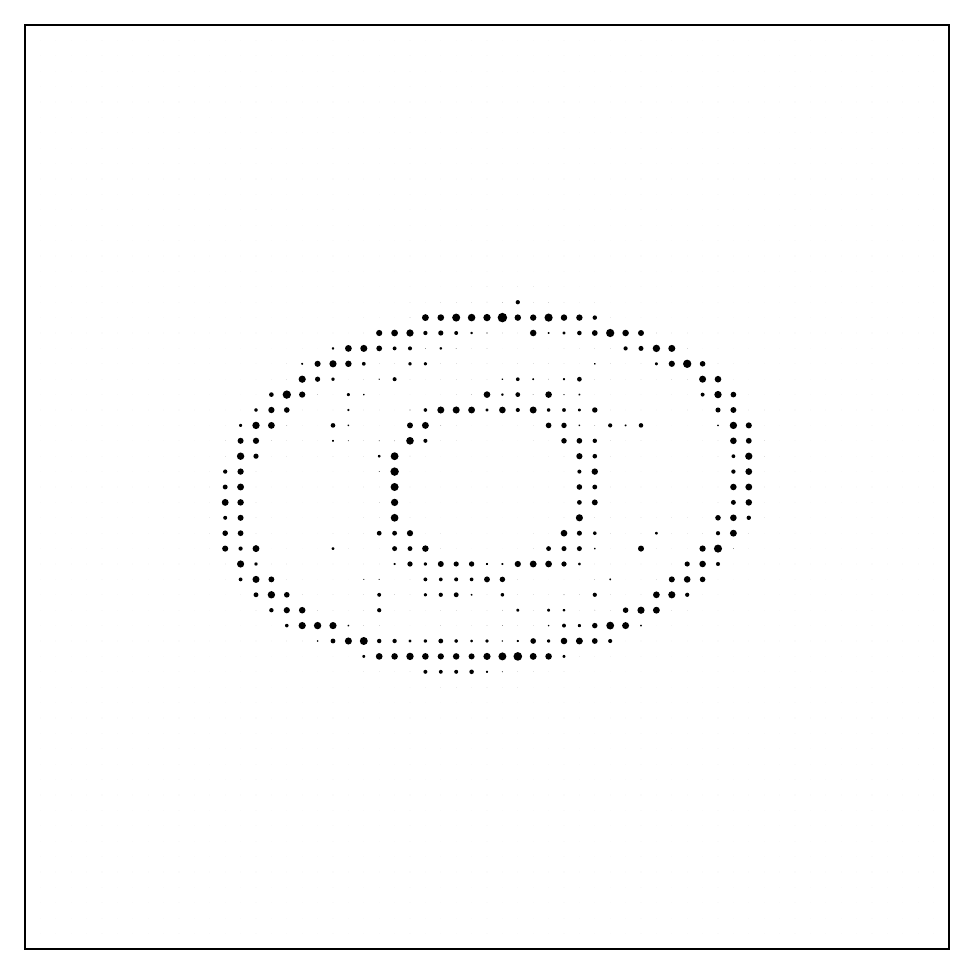}
			\end{tabular} &
			\begin{tabular}{@{}c@{}}\includegraphics[scale=0.25]{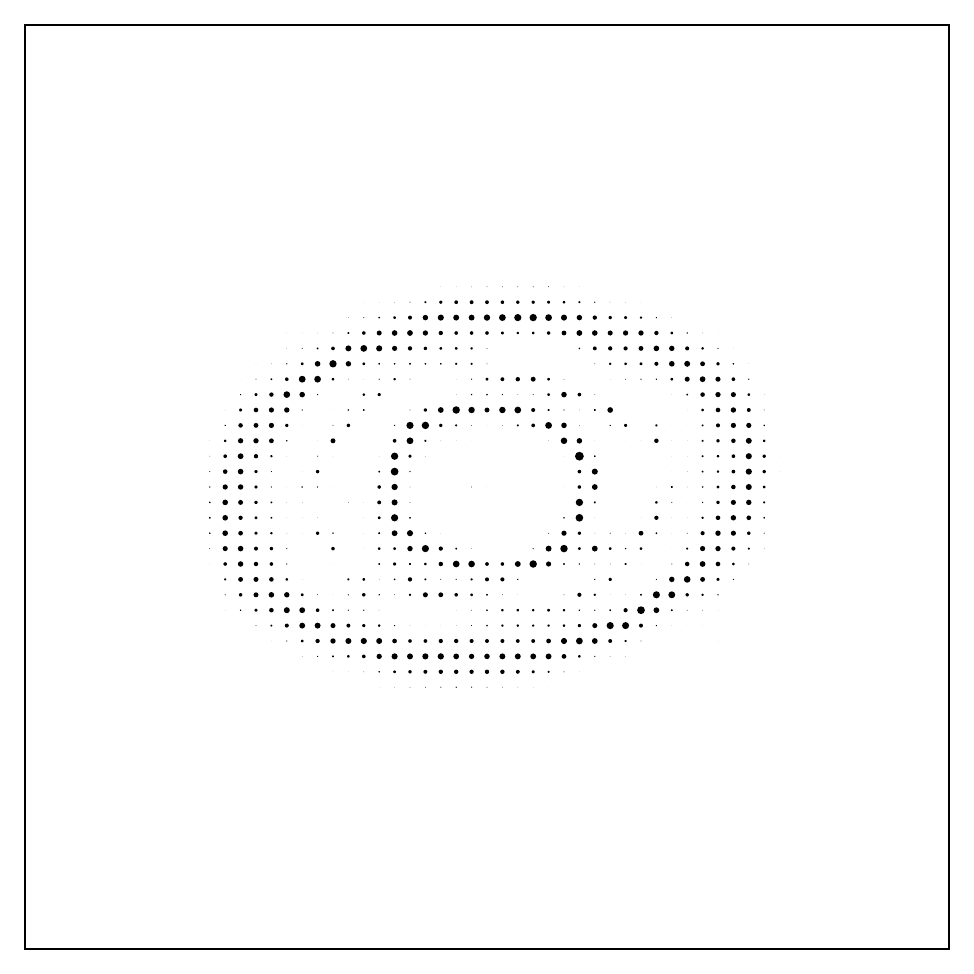}
			\end{tabular} &
			\begin{tabular}{@{}c@{}}\includegraphics[scale=0.25]{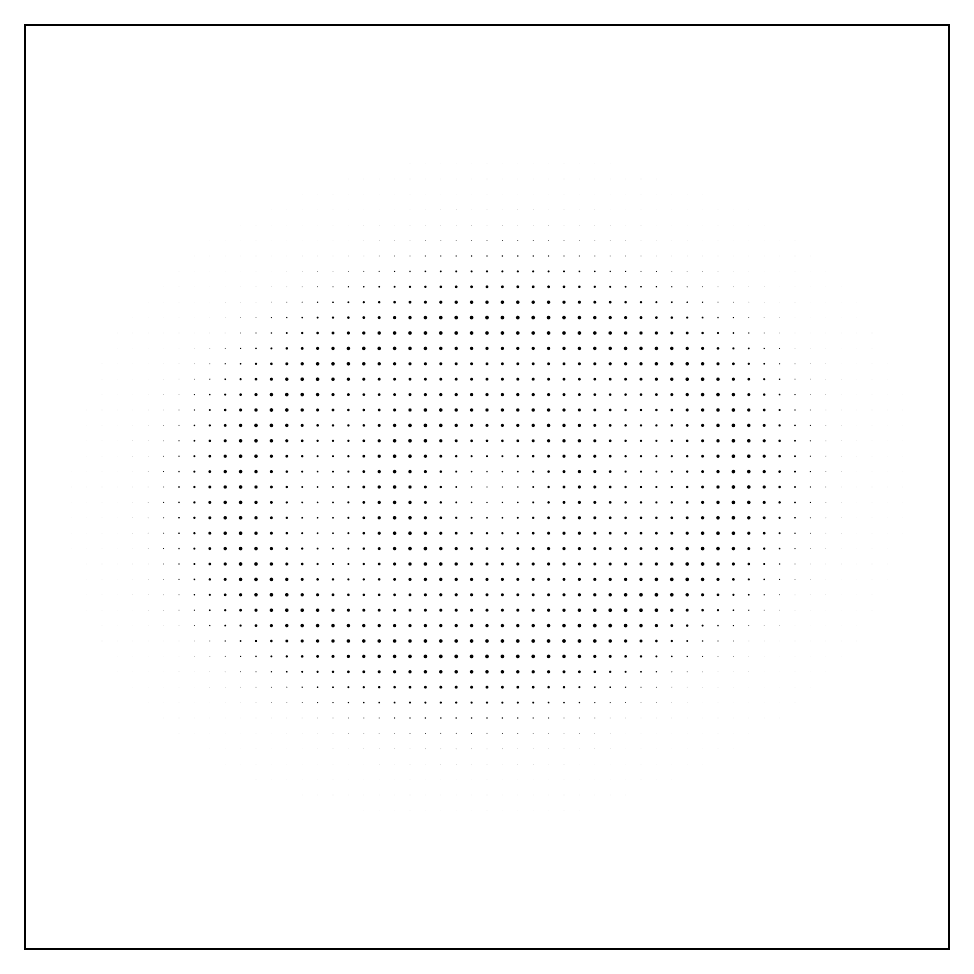}
			\end{tabular} &
			\begin{tabular}{@{}c@{}}\includegraphics[scale=0.25]{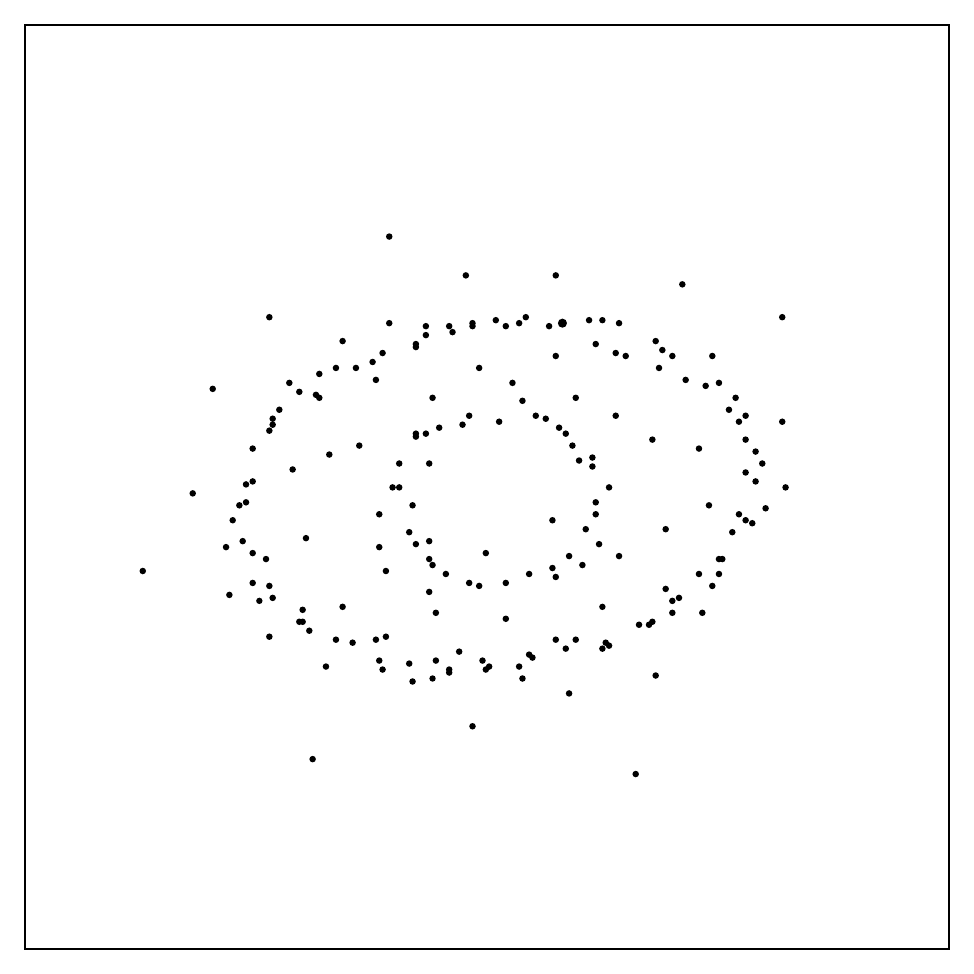}
			\end{tabular}
			\\
			
			Cost: 0.2666 (exact) & Cost: 0.2671 & Cost: 0.2675 & Cost: 0.2723 & Cost: 0.2790
	\end{tabular}}
	\caption{Comparison of barycenter algorithms on a standard benchmark dataset of ellipse images. Each barycenter atom is plotted as a disk with area proportional to its probability mass. All compared methods are run with the code, parameter choices, and dataset of \citep{janati2020debiased}. 
	}\label{fig:ellipsesbarycenters}
\end{figure}

\par Of the compared algorithms, MAAIPM gives the most accurate barycenter approximation. It uses a $60 \times 60$ grid fixed-support assumption. Although MAAIPM solves this fixed-support problem exactly, the support of an optimal barycenter does not lie on a $60 \times 60$ grid, and thus MAAIPM only computes an approximate barycenter. A natural approach is to run MAAIPM on a finer grid discretization, i.e., finer than $60 \times 60$.
However, this does not work, since MAAIPM does not scale to much larger grid sizes (see also Figure~\ref{fig:exp:ipm}).

The other two fixed-support algorithms are based on entropic regularization: debiased Sinkhorn barycenters \citep{janati2020debiased} and IBP \citep{solomon2015convolutional}. These use entropic regularization parameter $\gamma = 0.002$ and the same $60 \times 60$ fixed-support approximation as MAAIPM. Again, these methods produce suboptimal barycenters. While these methods scale to larger grid sizes than MAAIPM, this results in qualitatively similar and blurry visualizations as in this $60 \times 60$ case due to the entropic regularization.

The final compared algorithm is the free-support algorithm of \citep{luise2019sinkhorn}, which is based on the Frank-Wolfe algorithm. Although this method does not make a fixed-support assumption, it still returns an approximate solution due to the approximate nature of the Frank-Wolfe algorithm.

\section{Discussion}\label{sec:disc}
Wasserstein barycenters are used in many applications despite the fact that fundamental questions about their computational complexity are open---in particular, it was previously unknown whether barycenters can be computed in polynomial time. This paper addresses this issue by giving the first algorithm that, in any fixed dimension $d$, solves the barycenter problem exactly or to high precision in polynomial time.

\par Now, while our result answers the polynomial-time computability of barycenters from a \emph{theoretical} perspective, from a \emph{practical} perspective it is still a hard and interesting problem to compute high-precision barycenters for large-scale inputs. Indeed, our current implementation is not efficient beyond moderate-scale inputs; and while existing algorithms such as IBP scale to larger inputs, they have limited accuracy. Moreover, all existing algorithms pay for the curse of dimensionality in one way or another. We emphasize that our implementation does not contain further optimizations or heuristics; it is an interesting direction for future work to investigate potential such options including pruning cutting planes, warm starts, and specially tailored algorithms for the power diagram intersections in \S\ref{ssec:alg:step2} (e.g., in $\R^2$ or $\R^3$, settings which commonly arise in image processing and computer graphics applications).

\par We remark that another research direction suggested by the results of this paper is the possibility of solving ``structured'' Multimarginal Optimal Transport problems in $\poly(n,k)$ time despite the fact that they have $n^k$ exponentially many variables. Such problems arise in a variety of applications throughout data science and applied mathematics. However, the context of the application results in the cost tensor $C$ having different types of ``structure'', which in turn necessitates developing different techniques for efficiently solving the corresponding separation oracle (see the discussion in \S\ref{sec:alg}). We pursue this direction and develop such techniques in upcoming work~\citep{AltBoi20mot}.

\subsection{Extension to Wasserstein geometric median}\label{ssec:dis:ext}

We conclude the paper by briefly mentioning that the techniques we develop in this paper extend to solving several related problems. In particular, this gives the first polynomial-time algorithms for computing the \emph{Wasserstein geometric median}
\begin{align}
	\inf_{\nu} \sum_{i=1}^k \lambda_i \rho(\mu_i,\nu)
	\label{eq:geomedian}
\end{align}
of probability measures $\mu_1, \dots, \mu_k$, where $\rho$ is the $1$-Wasserstein distance $\rho$ (a.k.a., Earth Mover's Distance) over any of the popular ground metrics $\ell_{1}$, $\ell_2$, or $\ell_{\infty}$. See, e.g.,~\citep{Bac14} and the references within for background on this problem and its applications.

\par For brevity, we just state this result for exact computation (i.e., the analog of Theorem~\ref{thm:exactbary}). The high-precision analog of Theorem~\ref{thm:mainbary} also holds analogously. 

\begin{theorem}\label{thm:gm}
	Consider the space $\R^d$ endowed with any of the ground metrics $\ell_1$, $\ell_2$, or $\ell_{\infty}$.
	There is an algorithm that, given $k$ distributions each supported on $n$ atoms in $\R^d$ and a weight vector $\lambda$, computes an exact Wasserstein geometric median in $\poly(n,k,\log U)$ time, where $\log U$ is the bits of precision in the input. Moreover, this solution has support size at most $nk-k+1$. 
\end{theorem}

\par Since the proof of this extension is nearly identical, we sketch only the differences. First, observe that this geometric median problem~\eqref{eq:geomedian} is identical to the Wasserstein barycenter problem~\eqref{eq:intro:bary}, except that it depends on the $1$-Wasserstein distance $\rho$ rather than the squared $2$-Wasserstein distance $\cW$. It can be shown that the geometric median problem admits an analogous LP formulation as a Multimarginal Optimal Transport problem, just as in~\eqref{MOT-P}, except that here the cost tensor $C \in \Rntk$ has entries
\begin{align}
	C_{\jvec} = \min_{y \in \R^d} \sum_{i=1}^k \lambda_i c(x_{i,j_i},y)
	\label{C-geomedian}
\end{align}
where $c(\cdot,\cdot)$ denotes the relevant ground metric. This identical to the cost~\eqref{C-WASS} for barycenters, except that the squared Euclidean distance is replaced by the ground metric $c(\cdot,\cdot)$.

\par Now, since step $1$ of our algorithm---i.e., reducing finding a vertex solution of~\eqref{MOT-P} to solving the separation oracle for~\eqref{MOT-D}---holds for any Multimarginal Optimal Transport problem (see the discussion in \S\ref{sec:alg}), it remains only to adapt step $2$ of our algorithm. That is, it suffices to show that $\SEP$ can be efficiently implemented for the cost $C$ in~\eqref{C-geomedian}.

\par This extension requires only one small change: replace power diagrams with the analogous partitions of space that are defined with $\|x-y\|^2$ replaced by $c(x,y)$. These diagrams are sometimes called ``additively-weighted Voronoi diagrams with metric $c$''; for shorthand, we just call them ``$c$-diagrams'' here.

\begin{defin}[$c$-diagram]\label{def:diag-c}
	The $c$-diagram for points $z_1,\dots, z_n \in \R^d$ and radii $r_1, \dots, r_n \geq 0$ is the cell complex whose cells $E_1, \dots, E_n$ are given by
	\begin{align*}
		E_j = \{y \in \RR^d \, : \, c(z_j,y) - r_j^2 < c(z_{j'},y) - r_{j'}^2, \, \forall j' \neq j\}.
	\end{align*}
\end{defin}
Specifically, adapt the definition~\eqref{eq:Eij} of the sets $E_{i,j}$ in this way to $\{y \in \RR^d : c(x_{i,j},y) - [w_i]_j < c(x_{i,j'},y) - [w_i]_{j'}, \, \forall j' \neq j \}$. It is straightforward to check that our algorithm for the $\SEP$ oracle in \S\ref{ssec:alg:step2} and its proof then extend unchanged so long as the ground metric $c(\cdot,\cdot)$ satisfies the following basic properties, the first three of which are somewhat trivial but needed for rigor:
\begin{itemize}
	\item[(i)] There is a polynomial-time algorithm for evaluating $c(x,y)$ given points $x$ and $y$. 
	\item[(ii)] There is a polynomial-time algorithm for evaluating the geometric median $\min_{y \in \R^d} \sum_{i=1}^k \lambda_i c(x_i,y)$ given points $x_1, \dots, x_k$ and weights $\lambda_1,\dots, \lambda_k$.
	\item [(iii)] The closure of the cells in any $c$-diagram covers $\R^d$. 
	\item[(iv)] The intersection of any $k$ $c$-diagrams on $n$ points has $\poly(n,k)$ many non-empty subsets that can be enumerated in polynomial time.
\end{itemize}

The first three properties are trivially satisfied by all of the ground metrics $c(x,y) = \|x-y\|_1$, $\|x-y\|_2$, and $\|x-y\|_{\infty}$. Therefore in order to prove Theorem~\ref{thm:gm}, it remains only to verify property (iv).
For $\ell_1$ and $\ell_{\infty}$, $c$-diagrams are cell complexes with polynomially many affine facets which can moreover be computed in polynomial time \citep{klein1989concrete}, and thus the claim follows by Lemma~\ref{lem:hyperplaneintersectionenumeration}.
For $\ell_2$, the corresponding $c$-diagram is an additively-weighted Voronoi diagram, for which the desired complexity bounds are also known (see section 6.4 of~\citep{aurenhammer1987power}).

\section*{Acknowledgements.}
We thank Sinho Chewi, Jonathan Niles-Weed, and Pablo Parrilo for helpful conversations.

\footnotesize
\bibliographystyle{abbrv}
\bibliography{bary}{}

\end{document}

%% file: bary_preamble.tex
\usepackage{amsmath,amssymb,amsthm}
\usepackage{MnSymbol} 
\usepackage{graphicx,color}
\usepackage[hyphens]{url}
\usepackage{dsfont}
\usepackage{booktabs}
\usepackage[square,sort,numbers]{natbib}
\usepackage[sort,nocompress,noadjust]{cite}
\usepackage[normalem]{ulem}
\usepackage{mathtools}
\usepackage[nameinlink,capitalize]{cleveref}
\usepackage{multirow}
\usepackage{algorithm}
\usepackage[noend]{algpseudocode}
\usepackage{xspace}
\usepackage{tikz}

\usepackage{tikz-cd}
\usepackage[font=small,labelfont=bf]{caption}
\usepackage{subcaption}

\numberwithin{equation}{section}
\newtheorem{theorem}{Theorem}[section]

\newtheorem{lemma}[theorem]{Lemma}

\newtheorem{prop}[theorem]{Proposition}
\newtheorem{obs}[theorem]{Observation}

\newtheorem{defin}[theorem]{Definition}

\newcommand{\cH}{\mathcal{H}}

\newcommand{\cK}{\mathcal{K}}

\newcommand{\cM}{\mathcal{M}}

\newcommand{\cP}{\mathcal{P}}

\newcommand{\cW}{\mathcal{W}}

\newcommand{\R}{\mathbb{R}}

\newcommand{\Rn}{\R^n}

\newcommand{\Rntk}{(\R^n)^{\otimes k}}
\newcommand{\Rpntk}{(\R_{\geq 0}^n)^{\otimes k}}

\newcommand*{\poly}{\mathrm{poly}}
\newcommand*{\polylog}{\mathrm{polylog}}

\newcommand*{\eps}{\varepsilon}

\DeclareMathOperator*{\argmin}{argmin}

\renewcommand{\leq}{\leqslant}
\renewcommand{\geq}{\geqslant}

\newcommand{\MOTp}{\textnormal{(MOT$'$)}}
\newcommand{\MOTS}{\textnormal{(MOT}_S\textnormal{)}}

\newcommand{\SEP}{\textsf{SEP}}

\newcommand{\Coup}{\cM(\mu_1,\dots,\mu_k)}

\newcommand{\jvec}{\vec{j}}
\newcommand{\lvec}{\vec{\ell}}

\let\baraccent=\= %
\renewcommand{\=}[1]{\stackrel{#1}{=}} %

\providecommand{\RR}{\mathbb{R}}

\providecommand{\cM}{\mathcal{M}}

\providecommand{\cP}{\mathcal{P}}

\providecommand{\eps}{\epsilon}

\providecommand{\NP}{\mathsf{NP}}

\mathchardef\mhyphen="2D %

\providecommand{\sm}{\setminus}

\providecommand{\cH}{\mathcal{H}}

\newcommand{\interior}[1]{%
	{\kern0pt#1}^{\mathrm{o}}%
}

\usepackage{fancybox}
\newenvironment{fminipage}{\begin{Sbox}\begin{minipage}}{\end{minipage}\end{Sbox}\fbox{\TheSbox}}